\documentclass[10pt,a4paper,twoside]{amsart}

\usepackage{amsfonts, amssymb, amsmath, amsthm}
\usepackage{mathrsfs} 
\usepackage{latexsym}
\usepackage{enumerate}
\usepackage{multicol, relsize}
\usepackage{verbatim}

\usepackage{url}
\usepackage{csquotes}

\usepackage{scalerel,stackengine}
\stackMath
\newcommand\reallywidehat[1]{%
\savestack{\tmpbox}{\stretchto{%
  \scaleto{%
    \scalerel*[\widthof{\ensuremath{#1}}]{\kern-.6pt\bigwedge\kern-.6pt}%
    {\rule[-\textheight/2]{1ex}{\textheight}}
  }{\textheight}%
}{0.5ex}}%
\stackon[1pt]{#1}{\tmpbox}%
}

\usepackage[colorlinks=true]{hyperref}
\hypersetup{citecolor=blue}
\usepackage{mathabx} 

\newtheorem{thm}{Theorem}[section]
\newtheorem{lem}[thm]{Lemma}
\newtheorem{cor}[thm]{Corollary}
\newtheorem{prop}[thm]{Proposition}
\newtheorem{defi}[thm]{Definition}
\theoremstyle{remark}
\newtheorem{remark}[thm]{Remark}

\newcommand{\Z}[1]{\mathbb{Z}/#1\mathbb{Z}}

\DeclareMathOperator{\Id}{Id}

\newcommand{\subg}[1]{\langle{#1}\rangle{}}

\def\Ocal{\mathcal{O}}

\def\1{1\!\!\!1}

\def\mod{\mathbin{\,\textrm{mod}\,}}

\begin{document}

\title{Fast multi-precision computation of some Euler products
}

\author{S. Ettahri}
\address{Aix Marseille Univ, CNRS, Centrale Marseille, I2M, Marseille, France}
\email{salma.ettahri@etu.univ-amu.fr}

\author{O. Ramar\'e}
\address{CNRS / Aix Marseille Univ. / Centrale Marseille, I2M,
  Marseille, France}
\email{olivier.ramare@univ-amu.fr}

\author{L. Surel}
\address{Aix Marseille Univ, CNRS, Centrale Marseille, I2M, Marseille, France}
\email{leon.surel@etu.univ-amu.fr}


\date{August 18th 2019}


\hypersetup{pageanchor=false}
\maketitle
\hypersetup{pageanchor=true}

\begin{abstract}
  (\texttt{File \jobname.tex})
  For every modulus $q\ge3$, we define a family of subsets
  $\mathcal{A}$ of the multiplicative group $(\Z{q})^\times$ for which
  the Euler product $\prod_{p\text{mod}q\in\mathcal{A}}(1-p^{-s})$ can be
  computed in double exponential time, where $s>1$ is some given real
  number. We provide a Sage script to do so, and extend our result to
  compute Euler products $\prod_{p\in\mathcal{A}}F(1/p)/G(1/p)$ where
  $F$ and $G$ are polynomials with real coefficients, when this
  product  converges absolutely. This enables us to give
  precise values of several Euler products intervening in Number Theory.
\end{abstract}


\section{Introduction}
At the beginning of our query lie two constants that appear in the
paper \cite{Fouvry-Levesque-Waldschmidt*18} by \'E. Fouvry,
C. Levesque and M. Waldschmidt. On following this paper, they are
\begin{equation}
  \alpha_0^{(3)}
  =\frac{1}{2^{1/2}3^{1/4}}\prod_{p\equiv
    2[3]}\biggl(1-\frac{1}{p^2}\biggr)^{-1/2}
\end{equation}
and
\begin{equation}
  \label{eq:12}
  \beta_0=\frac{3^{1/4}\sqrt{\pi}}{2^{5/4}}\frac{\log(2+\sqrt{3})^{1/4}}{\Gamma(1/4)}
  \prod_{p\equiv
    5,7,11[12]}\biggl(1-\frac{1}{p^2}\biggr)^{-1/2}.
\end{equation}
Both occur in number theory as densities. The number of integers $n$ of the shape
$n=x^2-xy+y^2$, where $x$ and $y$
are integers (these are the so-called Loeschian numbers, see sequence A003136 of \cite{OEIS}) is given by
\begin{equation}
  \label{eq:1}
  N(x)=\alpha_0^{(3)}\frac{x(1+o(1))}{\sqrt{\log x}}.
\end{equation}
This accounts for our interest in the first constant.
The second one occurs because the
number of Loeschian numbers that are also sums of two squares (see
sequence A301430 of \cite{OEIS}) is given
by
\begin{equation*}
  N'(x)=\beta_0\frac{x(1+o(1))}{(\log x)^{3/4}}.
\end{equation*}
The question we address here is devising a fast manner to compute the
intervening Euler products.  From sequence A301429 of \cite{OEIS}, we know that
$ \alpha_0^{(3)}=0.638909\ldots$ but we would like (much!) more
digits. Similarly it is known that $\beta_0=0.30231614235\ldots$.
\begin{thm}
  We have
  \begin{align*}
    \alpha_0^{(3)}=
    0.&63890\,94054\,45343\,88225\,49426\,74928\,24509\,37549\,75508\,02912
    \\&33454\,21692\,36570\,80763\,10027\,64965\,82468\,97179\,11252\,86643\cdots
  \end{align*}
  and
  \begin{align*}
    \beta_0=0.&30231\,61423\,57065\,63794\,77699\,00480\,19971\,56024\,12795\,18936
    \\&96454\,58867\,84128\,88654\,48752\,41051\,08994\,87467\,81397\,92727\cdots
  \end{align*}
\end{thm}
Our method is more general and allows one to
compute Euler products of the shape
\begin{equation*}
  \prod_{p\in\mathcal{A}\mod q}(1-p^{-s})^{-1}
\end{equation*}
for any $s$ with $\Re s>1$ and some subsets
$\mathcal{A}$ of $(\mathbb{Z}/q\mathbb{Z})^\times$. We use a set of
identities that lead to fast convergent formulae. The use of a similar
 formula for scientific computations can be found in 
\cite[equation
(15)]{Shanks*64b} by D. Shanks. This author's approach has been put in a general context
by P. Moree \& D. Osburn in
\cite[equation (3.2)]{Moree-Osburn*06}. On looking closely, we see
that an
accurate value of $\alpha_0^{(3)}$ already follows from this paper.
The formulae we prove have a wider reach, though they fail to exhaust
the problem.
The reader may want to read subsection~\ref{Moredetails} now to
understand the initial idea. In the simplest form, we produce a
formula that links for instance $\zeta(s;12,1)=\prod_{\substack{ p\equiv
    1[12]}}{(1-p^{-s})^{-1}}$
to $\zeta(2s;12,1)$. We then reuse this formula to change $2s$ in $4s$, and so
on, and we finally use $\zeta(2^rs;12,1)=1+\Ocal(1/2^{s2^r})$. This is
analogous to D. Shanks scheme in \cite{Shanks*64b}. In the general
case however, we link values at $s$ with values at $ds$ for some $d>1$,
but these values are not the one of the same function, but of some
\emph{companion} functions. This means that we  have to work
simultaneously with several players. Let us first define these
\emph{companions}, which are all the products we propose to compute.

When $K$ is a cyclic subgroup of $(\mathbb{Z}/q\mathbb{Z})^\times$, we
denote by $A(K)$ the set of elements $x$ from
$(\mathbb{Z}/q\mathbb{Z})^\times$ such that the subgroup $\subg{x}$
generated by $x$ is equal to $K$. We note that the sets $A(K)$, when
$K$ ranges though the set of cyclic subgroups of
$(\mathbb{Z}/q\mathbb{Z})^\times$, determine a partition of $(\mathbb{Z}/q\mathbb{Z})^\times$.
A subset
$\mathcal{A}$ of $(\mathbb{Z}/q\mathbb{Z})^\times$ is said to be a
\emph{lattice-invariant} class if it is of the form $A(K)$ for
some cyclic subgroup $K$ of
$(\mathbb{Z}/q\mathbb{Z})^\times$, i.e.
if all its elements generate the same
subgroup (see Definition~\ref{li} below).
Here is a consequence of our approach.
\begin{thm}
  \label{mainspe}
  Let $q$ be some modulus and $\mathcal{A}$ be a
  \emph{lattice-invariant} class of
  $(\mathbb{Z}/q\mathbb{Z})^\times$. For every $s>1$, the product
  \begin{equation*}
    \zeta(s;q,\mathcal{A})=\prod_{p\mod q\in\mathcal{A}}(1-p^{-s})^{-1}
  \end{equation*}
  can be computed in double-exponential time.
\end{thm}
This theorem applies in particular to $\mathcal{A}=\{1\}$
and to $\mathcal{A}=\{-1\}$ and this is enough to compute
$\beta_0$ and $\alpha_0^{(3)}$. The last section contains numerical examples.
The material of this paper has been used to write the script
\begin{center}
  \texttt{LatticeInvariantEulerProducts-02.sage}
\end{center}
which we shorten below in \texttt{LIEP.sage} and which can be found on the second author
website. We give some details about this script when developing the
proof below.

We produce in Proposition~\ref{count3} an explicit expression for the
number $|G^\sharp|=|\mathscr{G}|$ of lattice-invariant classes. Though
our formula is only a sum of non-negative summands that are
multiplicative expressions, its
order of magnitude is not obvious when $q$ has numerous prime
factors. We have for instance not been able to establish that
$|G^\sharp|\ll _\epsilon q^\epsilon$ (for every positive $\epsilon$)
though this was our initial guess.

\subsection*{Notation}
When $\mathcal{A}$ is a subset of $(\Z{q})^\times$, we define
$\subg{\mathcal{A}}$ to be the (multiplicative) subgroup generated by
$\mathcal{A}$, and when $\mathcal{A}=\{a\}$, we may shorten
$\subg{\{a\}}$ in $\subg{a}$.

When additionally $P\ge2$ is some real parameter, we define
\begin{equation}
  \label{eq:2}
  \zeta(s;q,\mathcal{A})=\prod_{\substack{p\mod q\in\mathcal{A},\\
      p\ge P}}(1-p^{-s})^{-1}.
\end{equation}
This is in accordance with the notation of Theorem~\ref{mainspe}. We
define further
\begin{equation}
  \label{eq:11}
  L_P(s,\chi)=\prod_{p\ge P}(1-\chi(p)/p^s)^{-1}.
\end{equation}
\subsection*{Precise statement of the main result}
Let $q>1$ be a modulus. Let $G_0$ be a subgroup of
$G=(\mathbb{Z}/q\mathbb{Z})^\times$ and let $G_0^\perp$ be the subgroup of
characters that take the value~1 on $G_0$. Let $s>1$ be a real
number and $P\ge2$ be a parameter. We shall compute directly the contribution of
 the primes $<P$.
We define, for any positive integer $t$:
\begin{equation}
  \label{defgammaGzerot}
  \gamma_s(G_0,t)=\log\prod_{\chi\in G_0^\perp}L_P(ts,\chi).
\end{equation}
The parameter $P$ has disappeared from our notation and the reader may
stick with $P=2$.
When $s$ is a real number, the number $\prod_{\chi\in
  G_0^\perp}L_P(ts,\chi)$ is indeed a positive real number
because, when $\chi$ belongs to $G_0^\perp$, so does $\overline{\chi}$.

We denote the set of
\emph{lattice-invariant} classes by $G^\sharp$ and the set of cyclic
subgroups of $G$ by $\mathscr{G}$. Both sets are in an obvious one-to-one
correspondence. We consider the vector
\begin{equation}
  \label{defGammaoft}
  \Gamma_s(t)=(\gamma_s(G_0,t))_{G_0\in\mathscr{G}}.
\end{equation}
The rows of  $\Gamma_s(t)$ are indexed by cyclic subgroups
of~$G$.
It is computed by the function
\texttt{GetGamma} of the script 
\texttt{LIEP.sage} from the values of the Hurwitz zeta
function. See the implementation notes below.
We next define
\begin{equation}
  \label{defVsoft}
  V_s(t)=\bigl(\log\zeta_P(s;q,\mathcal{A})\bigr)_{\mathcal{A}\in G^\sharp}.
\end{equation}
The rows of  $V_s(t)$ are indexed by
classes. We control the size of our vectors
with the norm
\begin{equation}
  \label{eq:8}
  \|W\|=\max_i{|W_i|}
\end{equation}
when $W$ is the vector of coordinates $W_i$.
We define the square matrix $M_1^{-1}$ by
\begin{equation}
  \label{defL}
  M_1^{-1}\big|_{i=\mathcal{A}, j=K}=
    \begin{cases}
      \mu(|\subg{\mathcal{A}}/K|)/|G/K|&\text{when $K\subset \subg{\mathcal{A}}$},\\
      0&\text{otherwise}
    \end{cases}
  \end{equation}
where $\mathcal{A}$ ranges $G^\sharp$ while $K$ ranges $\mathscr{G}$.
It is unusual to define a matrix by its inverse. In the natural course
of the proof, a matrix $M_1$ will occur, whose inverse is the one
above; it is computed in Proposition~\ref{compL}. The reader will
readily check that there are no circularity in our definitions.
Let us recall that the \emph{exponent}
of $G$ is the maximal order of an element in~$G$ and is denoted by
$\exp G$. To each divisor $d>1$ of $\exp G$, we associate the square matrix
$N_d$ whose columns and rows are indexed by cyclic subgroups of $G$
and whose entries are given by
\begin{equation}
  \label{defNd}
  N_d\big|_{i=B_0, j=B_1}=
    \sum_{\substack{K\subset B_0,\\ |KB_1/K|=d}}\mu(|B_0/K|).
\end{equation}
The sum is over subgroups $K$. The condition $|KB_1/K|=d$ can be replaced
by the condition $|B_1/K\cap B_1|=d$.
Here is our main theorem.
\begin{thm}\label{precise}
  For any integer $r\ge2$, we have
  \begin{multline}
  \label{fineq}
  \biggl\|  V_s(1)-\sum_{0\le v\le r-1}(-1)^v
  \sum_{\substack{d_1\cdots d_{v}\le 2^r}}
  \frac{N_{d_1}}{d_1}\ldots
  \frac{N_{d_{v}}}{d_{v}}M_1^{-1}\Gamma_s(d_1\ldots d_v)
  \biggr\|
  \\\le
  \frac{1}{2}\biggl(1+\frac{r-1}{|G^\sharp|}\biggr)\biggl(\frac{|G^\sharp|d(\exp G)}{2}\biggr)^{r-1}
  \frac{1+P/(s2^r-1)}{P^{s2^r}}
\end{multline}
where $d_1,\ldots, d_r$ are all divisors of $\exp G$ excluding~1.
\end{thm}
When $v=0$, we use $d_1\ldots d_v=1$ and $N_{d_1}\cdots N_{d_v}=\Id$.
We provide in Section~\ref{mod7} the numerical datas modulo~7 that
will enable the reader to follow the proof step by step in this
case. This example may also be used to check our routines.

\subsection*{Extending the computations}
Now that we know how to compute some Euler products
$\zeta_P(s;q,\mathcal{A})$ in a fast manner, we can extend these
computations to more general Euler products, though still on the same
sets of primes. To do so, we add a definition:
\begin{equation}
  \label{defzetaPsqAr}
  (\log\zeta_P(s;q,\mathcal{A}|r))_{\mathcal{A}\in G^\sharp}
  =\mkern-10mu
  \sum_{0\le v\le r-1}\mkern-10mu(-1)^v\mkern-15mu
  \sum_{\substack{d_1\cdots d_{v}\le 2^r}}\mkern-5mu
  \frac{N_{d_1}}{d_1}\ldots
  \frac{N_{d_{v}}}{d_{v}}M_1^{-1}\Gamma_s(d_1\ldots d_v)
\end{equation}
\begin{thm}
\label{PM1}
  Let $F,G\in \mathbb R[X]$ be two coprime polynomials satisfying
$F(0)=G(0)=1$ such that $(F(X)-G(X))/X^2\in\mathbb{R}[X]$. Let $\beta\ge2$ be
an upper bound for the maximum modulus of the inverses of the roots of $F$ and of $G$.
Let $P\ge 2\beta$ be a parameter.  Then, for any parameters $J\ge3$
and $r\ge2$, we have
\begin{equation*}
\prod_{\substack{p\ge P,\\ p\in\mathcal{A}}}\frac{F(1/p)}{G(1/p)}
=\prod_{2\le j\le J}
\zeta_{P}(j;q,\mathcal{A}|r)^{b_G(j)-b_F(j)}\times I,
\end{equation*}
where the integers $b_G(j)$ and $b_F(j)$ are defined 
in Lemma~\ref{Wittpoly} and
\begin{equation*}
    |\log I|\le 
    \max(\deg F,\deg G)
  \biggl(
  \biggl(\frac{|G^\sharp|d(\exp G)}{2}\biggr)^{r-1}\frac{r\beta^2}{P^{2^{r+1}}}
  (1+2^{-r}P)
  +  
  \frac{4\beta^{J+1}}{JP^{J}}
    \biggr).
\end{equation*}
\end{thm}
\begin{remark}
  Inequality~\eqref{precbound} gives a more precise bound for $|\log
  I|$ which we will use in the actual script.
\end{remark}
\begin{remark}
  Lemma~\ref{easybeta} ensures that we may select
\begin{equation*}
  \beta = \max\Bigl(1, \sum_{1\le k\le\deg F}|a_k|, \sum_{1\le k\le\deg G}|b_k|\Bigr)
\end{equation*}
when $F(X)=1+a_1X+\ldots+a_\delta X^\delta$ and $G=1+b_1X+\ldots+b_{\delta'} X^{\delta'}$.
\end{remark}
\begin{remark}
  The function \texttt{GetEulerProds(q, F, G, nbdecimals)} gives all
  these Euler products. The polynomials $F$ and $G$ are to be given as
  polynomial expressions with the variable~$x$.
\end{remark}
D. Shanks in \cite{Shanks*60} (resp. \cite{Shanks*61},
resp. \cite{Shanks*67}) has already been able to compute an Euler
product over primes congruent to~1 modulo~8 (resp. to~1 modulo~4, resp. 1 modulo~8), by
using an identity (Lemma of section 2 for \cite{Shanks*60},
equation~(5) in \cite{Shanks*61} and the Lemma of section~3
in~\cite{Shanks*67}) that is a precursor of our Lemma~\ref{Wittpoly}.

In these three examples, the author has only been able to compute the
first five digits, and this is due to three facts: the lack of
interval arithmetic package at that time, the relative weakness of the
computers and the absence of a proper study concerning the error term.
We thus complement these results by giving the first hundred decimals.
\begin{cor}[Shank's Constant]
  We have
  \begin{equation*}
    \prod_{p\equiv
      1[8]}\biggl(1-\frac{4}{p}\biggr)\biggl(\frac{p+1}{p-1}\biggr)^2
    =\begin{aligned}[t]
      0.&95694\,53478\,51601\,18343\,69670\,57273\,89182\,87531
      \\&74977\,29139
      14789\,05432\,60424\,60170\,16444\,88885
      \\&94814\,40512\,03907\,95084\cdots
\end{aligned}
  \end{equation*}
  And thus Shank's constant satisfies
  \begin{align*}
    I
    &= \frac{\pi^2}{16\log(1+\sqrt{2})}\prod_{p\equiv
      1[8]}\biggl(1-\frac{4}{p}\biggr)\biggl(\frac{p+1}{p-1}\biggr)^2
    \\&=\begin{aligned}[t]
      0.&66974\,09699\,37071\,22053\,89224\,31571\,76440\,66883\,
      70157\,43648
      \\& \,24185\,73298\,52284\,52467\,99956\,45714
      \,72731\,50621\,02143\,59373\cdots
      \end{aligned}
  \end{align*}
\end{cor}
As explained in \cite{Shanks*61}, the number of primes $\le X$ of the form
$m^4+1$ is conjectured to be asymptotic to $I\cdot
X^{1/4}/\log X$.
The name ``Shank's Constant'' comes from Chapter 2, page 90 of
\cite{Finch*03}.
When using the script that we introduce below, this value is obtained with the call
\begin{center}
\texttt{GetEulerProds(8, $1-2*x-7*x^2-4*x^3$, $1-2*x+x^2$, 150, 400)}.
\end{center}
\begin{cor}[Lal's Constant]
  We have
  \begin{equation*}
    \prod_{p\equiv 1[8]}\frac{p(p-8)}{(p-4)^2}
    =
    \begin{aligned}[t]
      0.&88307\,10047\,43946\,67141\,78342\,99003\,10853\,46768
    \\&88834\,88097
    \,34707\,19295\,15939\,52119\,46990\,65659
    \\&68857\,99383\,28603\,79164\cdots
    \end{aligned}
  \end{equation*}
  And thus Lal's constant satisfies
  \begin{align*}
    \lambda
    &= \frac{\pi^4}{2^7\log^2(1+\sqrt{2})}\prod_{p\equiv
      1[8]}\biggl(1-\frac{4}{p}\biggr)^2\biggl(\frac{p+1}{p-1}\biggr)^4\prod_{p\equiv
      1[8]}\frac{p(p-8)}{(p-4)^2}
    \\&=
    \begin{aligned}[t]
      0.&79220\,82381\,67541\,66877\,54555\,66579\,02410\,11289\,32250\,98622
      \\&11172\,27973\,45256\,95141\,54944\,12490\,66029\,53883\,98027\,52927\cdots
    \end{aligned} 
  \end{align*}
\end{cor}
As explained in \cite{Shanks*67}, the number of primes $\le X$ of the form
$(m+1)^2+1$ and such that $(m-1)^2+1$ is also a prime is conjectured to be asymptotic to $\lambda\cdot
X^{1/2}/(\log X)^2$.
The name ``Lal's Constant'' comes from the papers \cite{Lal*67} and \cite{Shanks*67}.
When using the script that we introduce below, the first value is obtained with the call
\begin{center}
\texttt{GetEulerProds(8, $1-8*x$, $1-8*x+16*x^2$, 100, 400)}.
\end{center}

\bigskip
We close this section by mentioning another series of challenging
constants. In \cite{Moree*04b}, P. Moree computes inter alia the series of constants $A_\chi$
defined six lines after Lemma~3, page~452, by
\begin{equation}
  \label{eq:3}
  A_\chi= \prod_{p\ge2}\biggl(1+\frac{(\chi(p)-1)p}{(p^2-\chi(p))(p-1)}\biggr)
\end{equation}
where $\chi$ is a Dirichlet character. Our theory applies only when $\chi$ is real valued,


\subsection*{Thanks}
The authors thank M. Waldschmidt for having drawn their attention of
this question, P. Moree and \'E. Fouvry for helpful discussions on how
to improve this paper and X. Gourdon for free exchanges concerning some earlier computations.

\section{A general mechanism}

We start by presenting the mechanism of Shanks in \cite{Shanks*64}
is a general setting.

\begin{lem}\label{shanks}
  Let $\mathcal{P}$ be a set of prime numbers and let $f$ be a 
  function from $\mathcal{P}$ to $\{\pm1\}$. For every $s$
  with $\Re s>1$, we have
  \begin{equation*}
    \prod_{\substack{p\in\mathcal{P},\\ f(p)=-1}}(1-p^{-s})^2
    =
    \frac{\prod_{p\in\mathcal{P}}(1-p^{-s})}{\prod_{p\in\mathcal{P}}(1-f(p)p^{-s})}
    \prod_{\substack{p\in\mathcal{P},\\f(p)=-1}}(1-p^{-2s}).
  \end{equation*}
\end{lem}

\begin{proof}
  The proof is straightforward. We simply write
\begin{align*}
  \prod_{\substack{p\in\mathcal{P},\\
  f(p)=-1}}\frac{(1-p^{-s})^2}{1-p^{-2s}}
  &=
    \prod_{\substack{p\in\mathcal{P},\\ f(p)=-1}}\frac{1-p^{-s}}{1+p^{-s}}
  =
    \prod_{\substack{p\in\mathcal{P},\\ f(p)=-1}}\frac{1-p^{-s}}{1-f(p)p^{-s}}
  \\&=
    \prod_{\substack{p\in\mathcal{P}}}\frac{1-p^{-s}}{1-f(p)p^{-s}}
\end{align*}
as required.
\end{proof}
Shanks's method is efficient to deal with product of primes
belonging to a coset modulo a quadratic character. We generalize it as follows.
\begin{lem}\label{dede}
  Let $q>1$ be a modulus. We set $G_0$ be a subgroup of
  $G=(\mathbb{Z}/q\mathbb{Z})^\times$ and $G_0^\perp$ be the subgroup of
  characters that take the value~1 on $G_0$. For
  any integer $b$, we define $\subg{b}$ to the the subgroup generated
  by~$b$ modulo~$q$. We have
  \begin{equation*}
    \prod_{\chi\in G_0^\perp}L_P(s,\chi)
    =
    \prod_{\substack{G_0\subset K\subset G}}
    \prod_{\substack{p\ge P,\\ \subg{p}{} G_0=K}}
    \Bigl(1-p^{-|K/G_0|s}\Bigr)^{-|G/K|}
  \end{equation*}
  and, for any element $a\notin G_0$ of order~2, we have
  \begin{equation*}
    \prod_{\chi\in G_0^\perp}L_P(s,\chi)^{\chi(a)}
    =
    \prod_{\substack{G_0\subset K\subset G,\\ a\in K}}
    \prod_{\substack{p\ge P,\\ \subg{p}{} G_0=K}}
    \biggl(\frac{(1-p^{|K/G_0|s/2})^2}{1-p^{-|K/G_0|s}}\biggr)^{-|G/K|}
  \end{equation*}
  where $\hat{G}$ is the set of characters of $G$.
\end{lem}
Case $G_0=\{1\}$ of the first identity is classical in Dedekind zeta
function theory, and can be found in~\cite[Proposition 13]{Serre*70}
in a rephrased form. Case $a\neq1$ will not be required for the
general theory. It may however lead quickly to efficient formulae.

\begin{proof}
  We note that $\prod_{\chi\in G_0^\perp}(1-\chi(p)z)^{\chi(a)} =\prod_{\psi\in\hat{H}}(1-\psi(p)z)^{f(\psi)}$
  when $\subg{p}=H$ and where
  \begin{equation}
    \label{eq:6}
    f(\psi)=\sum_{\substack{\chi\in G_0^\perp,\\ \chi|H = \psi}}\chi(a).
  \end{equation}
  The condition $\chi\in G_0^\perp$ can also be
  written as $\chi|G_0=1$, hence we can assume that $\psi|(H\cap
  G_0)=1$. We write
  \begin{equation*}
    \prod_{\chi\in G_0^\perp}(1-\chi(p)z)^{\chi(a)}
    =\prod_{\substack{\psi'\in\widehat{H{} G_0},\\ \psi'|G_0=1}}(1-\psi(p)z)^{f'(\psi')}
  \end{equation*}
  where
  \begin{equation}
    \label{eq:76}
    f'(\psi')=\sum_{\substack{\chi\in G_0^\perp,\\ \chi|H{} G_0 = \psi}}\chi(a).
  \end{equation}
  When $a$ lies outside $H{} G_0$, this sum vanishes;
  otherwise
  it equals $|G/(H{} G_0)|\psi'(a)$.
  The characters of $H{} G_0$ that are trivial on $G_0$ are
  canonically identified with the
  characters of the cyclic group $(H{} G_0)/G_0$.
  We thus have
  \begin{equation*}
    \prod_{\substack{\psi'\in\widehat{H{} G_0},\\ \psi'|G_0=1}}(1-\psi(p)z)
    =
    1-z^{|(H{} G_0)/G_0|}
  \end{equation*}
  and this proves our first formula.
  
  When $a^2\equiv 1[q]$ and $a\notin G_0$, and since $(H{} G_0)/G_0$
  is cyclic, of (even) order~$h$ say, the characters are given by
  $\chi(p^x)=e(cx/h)$ since $p$ is a generator and where $c$ ranges
  $\{0,\cdots,h-1\}$. We thus have, when $a\in H$,
  \begin{align*}
    \prod_{\psi'\in\reallywidehat{(H{} G_0)/G_0}}\mkern-30mu(1-\psi'(p)z)^{\psi'(a)}
    &=\mkern-5mu\prod_{c\mod h}(1-e(c/h)z)^{e(c/2)}
    \\&=\mkern-7mu\prod_{0\le d\le \frac{h-2}2}
    \biggl(1-e\Bigl(\frac{2d}{h}z\Bigr)\biggr)
    \prod_{0\le d\le \frac{h-2}2}\biggl(1-e\Bigl(\frac{2d+1}{h}z\Bigr)\biggr)^{-1}
    \\&=
    \frac{1-z^{h/2}}{1-(e(1/h)z)^{h/2}}
    =\frac{1-z^{h/2}}{1+z^{h/2}}
    =\frac{(1-z^{h/2})^2}{1-z^h}.
  \end{align*}
  The reader will readily complete the proof by setting $K=H{} G_0$.
\end{proof}

\subsection{A special case}

Let us select for $G_0$ the kernel of a given quadratic character
$\chi_1$. The subgroup $K$ can take only two values, $G_0$ or~$G$. We
thus get
\begin{equation*}
  L(s,\chi_1)L(s,\chi_0)=
  \prod_{\chi_1(p)=1}(1-p^{-s})^{2}
  \prod_{\chi_1(p)=-1}(1-p^{-2s})
\end{equation*}
which gets converted into
\begin{equation}\label{myshanks}
  L(s,\chi_1)L(s,\chi_0)=
  L(s,\chi_0)^{2}
  \prod_{\chi_1(p)=-1}(1-p^{-s})^{-2}
  \prod_{\chi_1(p)=-1}(1-p^{-2s}).
\end{equation}
Lemma~\ref{shanks} can also be used to obtain the same result.

\subsection{More details modulo 12}
\label{Moredetails}
Here is the character table modulo~12:

\vspace*{3mm}
\begin{tabular}{|c|cccc|}
  \hline
  &1&5&7&11\cr
  \hline
  $\chi_{0,12}$&1&1&1&1\cr        
  $\chi_{1,12}$&1&-1&1&-1\cr        
  $\chi_{2,12}$&1&1&-1&-1\cr        
  $\chi_{3,12}$&1&-1&-1&1\cr        
  \hline
\end{tabular}
\subsubsection*{First Identity} This table enables us to write:
\par
\vspace*{3mm}
\begin{tabular}{|c|cccc|}
  \hline
  $p\mod 12$&1&5&7&11\cr
  \hline
  $(1-\chi_{0,12}(p)z)$&$1-z$&$1-z$&$1-z$&$1-z$\cr        
  $(1-\chi_{1,12}(p)z)$&$1-z$&$1+z$&$1-z$&$1+z$\cr        
  $(1-\chi_{2,12}(p)z)$&$1-z$&$1-z$&$1+z$&$1+z$\cr        
  $(1-\chi_{3,12}(p)z)$&$1-z$&$1+z$&$1+z$&$1-z$\cr
  \hline
  $\prod_\chi\cdots$&$(1-z)^4$&$(1-z^2)^2$&$(1-z^2)^2$&$(1-z^2)^2$\cr
  \hline
\end{tabular}
\vspace*{3mm}
\par\noindent
And thus
\begin{equation*}
  \prod_{\chi}L(s,\chi)=\prod_{p\ge5}\frac{1}{(1-p^{-2s})^2}
  \prod_{\substack{p\ge5,\\ p\equiv 1[12]}}\frac{(1-p^{-2s})^2}{(1-p^{-s})^4},
\end{equation*}
which gives rise to the formula
\begin{equation*}
  \prod_{\substack{p\ge5,\\ p\equiv 1[12]}}\frac{1}{(1-p^{-s})^4}
  =
  \prod_{\substack{p\ge5,\\ p\equiv 1[12]}}\frac{1}{(1-p^{-2s})^2}
  \frac{\prod_{\chi}L(s,\chi)}{((1-2^{-2s})(1-3^{-2s})\zeta(2s))^{2}}.
\end{equation*}
This identity reduces the computation of $\zeta(s; 12,1)$
to the one of $\zeta(2s;12,1)$ and we can iterate this formula. Note that
we can take the required fourth root as only real numbers are
involved, when the terms are properly grouped.
\subsubsection*{Second Identity}
Similarly, we find that
\vspace*{3mm}
\par\noindent
\begin{tabular}{|c|cccc|}
  \hline
  $p$&1&5&7&11\cr
  \hline
  $(1-\chi_{0,12}(p)z)$&$1-z$&$1-z$&$1-z$&$1-z$\cr        
  $(1-\chi_{1,12}(p)z)^{-1}$&$(1-z)^{-1}$&$(1+z)^{-1}$&$(1-z)^{-1}$&$(1+z)^{-1}$\cr        
  $(1-\chi_{2,12}(p)z)^{-1}$&$(1-z)^{-1}$&$(1-z)^{-1}$&$(1+z)^{-1}$&$(1+z)^{-1}$\cr        
  $(1-\chi_{3,12}(p)z)$&$1-z$&$1+z$&$1+z$&$1-z$\cr
  \hline
  $\prod_\chi\cdots$&$1$&$1$&$1$&$\displaystyle\frac{(1-z)^4}{(1-z^2)^2}$\cr
  \hline
\end{tabular}
\vspace*{3mm}
\par\noindent
whence
\begin{equation*}
  \frac{L(s,\chi_{0,12})L(s,\chi_{3,12})}
  {L(s,\chi_{1,12})L(s,\chi_{2,12})}
  =
  \prod_{\substack{p\ge5,\\ p\equiv 11[12]}}\frac{(1-p^{-s})^2}{(1-p^{-2s})},
\end{equation*}
which we finally write in the form
\begin{equation*}
  \prod_{\substack{p\ge5,\\ p\equiv 11[12]}}\frac{1}{(1-p^{-s})^2}
  =
  \frac{L(s,\chi_{0,12})L(s,\chi_{3,12})}
  {L(s,\chi_{1,12})L(s,\chi_{2,12})}
  \prod_{\substack{p\ge5,\\ p\equiv 11[12]}}\frac{1}{(1-p^{-2s})^2}.
\end{equation*}
This identity again reduces the computation of $\zeta(s;12,11)$
to the one of $\zeta(2s;12,11)$ and we can iterate this formula. Again, 
we can take the required fourth rooths as only real numbers are
involved, when the terms are properly grouped. 

\subsubsection*{Third Identity}
We also find that
\vspace*{3mm}
\par\noindent
\begin{tabular}{|c|cccc|}
  \hline
  $p$&1&5&7&11\cr
  \hline
  $(1-\chi_{0,12}(p)z)$&$1-z$&$1-z$&$1-z$&$1-z$\cr        
  $(1-\chi_{1,12}(p)z)^{-1}$&$(1-z)^{-1}$&$(1+z)^{-1}$&$(1-z)^{-1}$&$(1+z)^{-1}$\cr        
  $(1-\chi_{2,12}(p)z)$&$1-z$&$1-z$&$1+z$&$1+z$\cr        
  $(1-\chi_{3,12}(p)z)^{-1}$&$(1-z)^{-1}$&$(1+z)^{-1}$&$(1+z)^{-1}$&$(1-z)^{-1}$\cr
  \hline
  $\prod_\chi\cdots$&$1$&$\displaystyle\frac{(1-z)^4}{(1-z^2)^2}$&$1$&$1$\cr
  \hline
\end{tabular}
\vspace*{3mm}
\par\noindent
whence
\begin{equation*}
  \frac{L(s,\chi_{0,12})L(s,\chi_{2,12})}
  {L(s,\chi_{1,12})L(s,\chi_{3,12})}
  =
  \prod_{\substack{p\ge5,\\ p\equiv 5[12]}}\frac{(1-p^{-s})^2}{(1-p^{-2s})},
\end{equation*}
We are exactly in the same position as with the second identity.
We again finally write in the form
\begin{equation*}
  \prod_{\substack{p\ge5,\\ p\equiv 5[12]}}\frac{1}{(1-p^{-s})^2}
  =
  \frac{L(s,\chi_{0,12})L(s,\chi_{2,12})}
  {L(s,\chi_{1,12})L(s,\chi_{3,12})}
  \prod_{\substack{p\ge5,\\ p\equiv 5[12]}}\frac{1}{(1-p^{-2s})^2}.
\end{equation*}
This identity again reduces the computation of $\zeta(s; 12,5)$
to the one of $\zeta(2s;12,5)$ and we can iterate this formula. Again, 
we can take the required fourth rooths as only real numbers are
involved, when the terms are properly grouped. 

\subsubsection*{Fourth Identity}
We can easily produce a similar formula linking
$\zeta(s; 12,7)$ to
$\zeta(2s; 12, 7)$ or
use the fact that the product
$\zeta(s; 12, 1)\zeta(s; 12, 5)\zeta(s; 12, 7)\zeta(s; 12, 11)$ equals
$L(s,\chi_{0,12})$, and thus is known, to infer such
a formula from the ones above.

\section{Products obtained in general}
We want to compute Euler products of the shape
$\zeta(s;q,\mathcal{A})$ for $s>1$ and some subset $\mathcal{A}$ of
$(\mathbb{Z}/q\mathbb{Z})^\times$. Computing $L(s,\chi)$ is easier as
it can be reduced to sums over integers is some arithmetic
progressions. Equation~\eqref{myshanks} reduces in a special case the
computations of $\zeta(s;q,\mathcal{A})$ to the one of
$\zeta(s;q,\mathcal{A})$, and we can continue the process. We soon
reach $\prod_{p\in\mathcal{A}\mod q}(1-1/p^{2^Ns})$ with a large
enough~$N$ which can be approximated by $1+\Ocal(2^{-2^Ns})$. The
object of this section is to devise a setting to understand which sums
we relate together.
\begin{defi}\label{li}\label{mono}
  Two elements $g_1$ and $g_2$ of the abelian group $G$ are said to be
  \emph{lattice-invariant}
if and only if they generates the same
  group.

  The map between the set of cyclic subgroups of $G$ and the set of
  \emph{lattice-invariant}-classes which, to a subgroup, associates
  the subset of its generators, is one-to-one.
  
\end{defi}
The function \texttt{GetLatticeInvariantClasses} of the script
\texttt{LIEP.sage} gives the two lists: the one of the cyclic
subgroups and the one of their generators, ordered similarly and in
increasing size of the subgroup.

Any two elements of
$(\Z{q})^\times$ equivalent according to it cannot be distinguished by
using the formulae of Lemma~\ref{dede}. Conversely, the question is
to know whether we are indeed able to distinguish each class.
%
%
To each class
$\mathcal{A}$, we attach the enumerable collection of symbols $(x_{\mathcal{A}}^r)_{r\ge1}$.
We shall replace each of them according to the rule
\begin{equation}
  \label{eq:9}
  x_{\mathcal{A}}^r \mapsto -\log\prod_{\substack{p+q\mathbb{Z}\in {\mathcal{A}},\\ p\ge P}}\bigl(1-p^{-rs}\bigr).
\end{equation}
We consider the module of
finite formal combinations
\begin{equation*}
  \sum_{\substack{{\mathcal{A}}\in G^\sharp,\\ r\ge1}}\alpha_{{\mathcal{A}},r} x_{\mathcal{A}}^r
\end{equation*}
with coefficients $\alpha_{{\mathcal{A}},r}\in\mathbb{Z}$ and indeterminates
$x_{\mathcal{A}}^r$. The superscript $r$ is \emph{not} a power.
We consider the following special elements. Let $G_0\subset K\subset
G$ be two subgroups such that $K/G_0$ is cyclic. We define
\begin{equation}
  \label{eq:5}
  g(G_0,K,t)=\sum_{{\mathcal{A}}\in G^\sharp, {\mathcal{A}}G_0=K}x_{{\mathcal{A}}}^{t|K/G_0|}.
\end{equation}
With that, we find that
\begin{equation}
  \label{basicrel}
  \gamma(G_0,t)=\sum_{\substack{G_0\subset K\subset G}}
  |G/K|g(G_0,K,t).
\end{equation}

\section{Iterating the formula}

The first identity of Lemma~\ref{dede} gives us as many identities as
there are subgroups~$G_0$; we know by Definition~\ref{mono} that the
number of \emph{lattice-invariant}-classes equals the one of cyclic
subgroups. It turns out that it is enough to restrict our attention to
cyclic subgroups $G_0$. Let $\mathscr{G}$ be the subset of such
subgroups, which we order by inclusion.  On
recalling definition~\eqref{defGammaoft}, we may
rewrite~\eqref{basicrel} in the form
\begin{equation}
  \label{initeq}
  \Gamma(t)
  =\sum_{d||G|}M_d V_s(dt)
\end{equation}
where (this is the case $K=G_0$)
\begin{equation}
  \label{eq:20}
  M_1\big|_{i=G_0, j={\mathcal{A}}}=
  \begin{cases}
    |G/K|&\text{if ${\mathcal{A}}\subset G_0$},\\
    0&\text{otherwise,}
  \end{cases}
\end{equation}
and, where, when $d>1$ (i.e. $G_0\subsetneq K$), we have
\begin{equation}
  \label{eq:19}
  M_d\big|_{i=G_0, j={\mathcal{A}}}=
  \begin{cases}
    |G/{\mathcal{A}}G_0|&\text{if $|{\mathcal{A}} G_0|/|G_0|=d$},\\
    0&\text{otherwise.}
  \end{cases}
\end{equation}
Equation~\eqref{initeq} gives us a relation between $M_1V_s(t)$ and
$M_dV_s(dt)$ for several $d$'s that are strictly larger than~1. Our
roadmap is to invert the matrix $M_1$ and to iterate this formula.
We compute explicitly $M_1^{-1}$ by using some generalised Moebius
inversion, which we first put in place.
\subsubsection*{The Moebius function associated to $\mathscr{G}$}
We follow closely the exposition of Rota in~\cite{Rota*64a}.
On the algebra of functions $f$ on couples $(K,L)$ of points of
$\mathscr{G}$ such that $K\subset L$ (the so-called \emph{incidence
  algebra}, see \cite[Section 3]{Rota*64a}), we define the convolution
product
\begin{equation*}
  (f\star g)(K,L)=\sum_{K\subset H\subset L}f(K,H)g(H,L).
\end{equation*}
We consider the $\mathscr{G}$-zeta function which is defined by
\begin{equation*}
  \zeta_{\mathscr{G}}(K,L)=
  \begin{cases}
    1&\text{when $K\subset L$,}\\
    0&\text{otherwise}.
  \end{cases}
\end{equation*}
This function is shown to be invertible in the above algebra and its
inverse is called the $\mathscr{G}$-Moebius function, denoted by $\mu_{\mathscr{G}}$.
By definition, we have the two Moebius inversion formulas:
\begin{equation}
  \label{eq:21}
  \sum_{K\subset H\subset L}f(K,H)=g(K,L)\implies
  f(K,L)=\sum_{K\subset H\subset L}g(K,H)\mu_{\mathscr{G}}(H,L)
\end{equation}
and
\begin{equation}
  \label{eq:21b}
  \sum_{K\subset H\subset L}f(H,L)=g(K,L)\implies
  f(K,L)=\sum_{K\subset H\subset L}\mu_{\mathscr{G}}(K,H)g(H,L).
\end{equation}
We end this reminder with a formula giving the value of
$\mu_{\mathscr{G}}(K,H)$.
\subsubsection*{Computing $\mu_{\mathscr{G}}(K,H)$}
Let $C_p(K,H)$ be the number of chains of length $p$ going from $K$ to
$H$, i.e. the number of $p+1$-uples $K=A_0\subsetneq A_1\subsetneq
A_2\subsetneq \ldots \subsetneq A_p=H$. Then (cf \cite[Proposition 6]{Rota*64a})
\begin{equation}
  \label{eq:22}
  \mu_{\mathscr{G}}(K,H)=\sum_{p\ge0}(-1)^p C_p(K,H).
\end{equation}
Since the subgroups of a cyclic group are all cyclic, we only have to
consider the chains in $H/K$. There is one and only one subgroup for
each divisor of $|H/K|$, and any two such subgroups $L_1$ and $L_2$
are included according to whether $|L_1|\big||L_2|$ or not. This
transfers the problem on a problem on integers. Let $c_{\ell}(n)$ be the
number of $\ell+1$-divisibility chains between 1 and~$n$. We have
$c_0(n)=\1_{n=1}$ while $c_1(n)=\1_{n\ge2}$ and $c_{p+1}(n)=(c_{\ell}\star c_1)(n)$. This proves that
$c_{\ell}(n)=d_{\ell}^*(n)$, the number of $p$-tuples $(d_1,d_2,\ldots,d_{\ell})$ of
divisors of~$n$ that are such that $d_i\neq1$ and $d_1d_2\cdots
d_{\ell}=n$. We have
\begin{equation*}
  \sum_{n\ge1}d_{\ell}^*(n)/n^s=(\zeta(s)-1)^\ell
\end{equation*}
and thus the generating series of $\sum_{p\ge0}(-1)^pd^*_{\ell}(n)$ is
\begin{equation*}
  \sum_{\ell\ge0}(-1)^\ell(\zeta(s)-1)^\ell=\frac{1}{1+\zeta(s)-1}=1/\zeta(s).
\end{equation*}
We have proved that
\begin{equation}
  \label{eq:23}
  \mu_{\mathscr{G}}(K,H)=\mu(|H/K|).
\end{equation}
\subsubsection*{Inverting the matrix $M_1$}

\begin{prop}
  \label{compL}
  The matrix $M_1$ is invertible and the coefficients of its inverse
  are given by
  \begin{equation*}
    M_1^{-1}\big|_{i={\mathcal{A}}, j=K}=
    \begin{cases}
      \mu(|\subg{{\mathcal{A}}}/K|)/|G/K|&\text{when $K\subset \subg{{\mathcal{A}}}$},\\
      0&\text{otherwise.}
    \end{cases}
  \end{equation*}
\end{prop}
\begin{proof}
  We find that
  \begin{equation*}
    M_1 V=(|G/K|\sum_{{\mathcal{A}}\subset K}v_{\mathcal{A}})_K.
  \end{equation*}
  We replace ${\mathcal{A}}$ by the subgroup $B=\subg{{\mathcal{A}}}$ it generates. Inverting
  $f(K)=|G/K|\sum_{B\subset K}v_B$ is done with the Moebius function of
  $\mathscr{G}$. To do so, simply consider the more general function
  \begin{equation*}
    F(H,K)=|G/K|\sum_{H\subset B\subset K}v^*(H,B)
    =|G/K|(v^*\star\zeta_{\mathscr{G}})(H,K)
  \end{equation*}
  where $v^*(H,B)=v_B$. This gets inverted in
  \begin{equation*}
    v^*(H,B)=\sum_{H\subset K\subset B}F(H,K)|G/K|^{-1}\mu_{\mathscr{G}}(K,B)
  \end{equation*}
  which yield, by specializing $H=\{1\}$
  \begin{equation*}
    v_B=\sum_{K\subset B}f(K)|G/K|^{-1}\mu_{\mathscr{G}}(K,B).
  \end{equation*}
  We could also have applied \cite[Proposition 2 (**)]{Rota*64a}.
  This gives us
  \begin{equation*}
    M_1^{-1}\big|_{i=B, j=K}=
  \begin{cases}
    \mu_{\mathscr{G}}(K,B)/|G/K|&\text{if $K\subset B$},\\
    0&\text{otherwise.}
  \end{cases}
\end{equation*}
Our proposition is proved.
\end{proof}
The function \texttt{GetM1Inverse} of the script \texttt{LIEP.sage}
computes $M_1^{-1}$.
\subsubsection*{The recursion formula}
We start from~\eqref{initeq} and deduce that
\begin{equation}
  \label{seceq}
  V_s(t)=-\sum_{\substack{d||G|,\\ d\neq1}}M_1^{-1}M_d V_s(dt)+M_1^{-1}\Gamma(t).
\end{equation}
We readily find that $N_d=dM_1^{-1}M_d$ is given by~\eqref{defNd}.
\begin{proof}
  Indeed we have
  \begin{equation*}
    N_d\big|_{i=B_0, j=B_1}=d\sum_{\substack{K\subset B_0,\\ K\subset
        B_1,\\ |KB_1/K|=d }}
    \mu(|B_0/K|)|G/K|^{-1}|G/B_1|.
  \end{equation*}
  This is exactly what we have written in~\eqref{defNd}.
\end{proof}
By considering the exact sequence
\begin{equation}
  \label{eq:10}
  1\longrightarrow K\cap B_1 \mathop{\xrightarrow{\hspace*{30pt}}}\limits_{k\mapsto (k,
    k^{-1})} K\times B_1\mathop{\xrightarrow{\hspace*{30pt}}}\limits_{ (k,b_1)\mapsto kb_1}
  KB_1 \longrightarrow 1,
\end{equation}
one shows that $|KB_1/K|=|B_1|/|K\cap B_1|$. As a consequence, we see that only the $d$ that divides the \emph{exponent}
of $G$ appear.
The function \texttt{GetNds} of the script
  \texttt{LIEP.sage}
computes $(N_d)_{d}$.

\vspace*{3mm}
\noindent\framebox{\parbox{\linewidth}{\begin{equation}
  \label{trieq}
  V_s(t)=-\sum_{\substack{d|\exp G,\\ d\neq1}}
  \frac{N_d}{d} V_s(dt)+M_1^{-1}\Gamma(t).
\end{equation}}}
\subsubsection*{Unfolding the recursion}
Let $z\ge1$ and $r\ge1$ be two parameters.
We have
\begin{multline}
  \label{recursion}
  V_s(t)=  
  (-1)^{r}
  \sum_{\substack{d_1\cdots d_{r}\le z}}
  \frac{N_{d_1}}{d_1}\ldots
  \frac{N_{d_r}}{d_r} V_s(d_1\ldots d_r t)
  \\+
  \sum_{1\le v\le r}(-1)^v
  \sum_{\substack{d_1\cdots d_{v-1}\le z,\\
      d_1\cdots d_{v-1}d_v>z}}
  \frac{N_{d_1}}{d_1}\ldots
  \frac{N_{d_v}}{d_v} V_s(d_1\ldots d_v t)
  \\
  +\sum_{1\le v\le r-1}(-1)^{v}
  \sum_{\substack{d_1\cdots d_{v}\le z}}
  \frac{N_{d_1}}{d_1}\ldots
  \frac{N_{d_{v}}}{d_{v}}M_1^{-1}\Gamma(d_1\ldots d_v t)
  +
  M_1^{-1}\Gamma(t)
\end{multline}
where $d_1,\ldots, d_r$ are all divisors of $\exp G$ excluding~1.
We can incorporate the last summand in the one before by considering
as the value for $s=0$.
\begin{proof}
  Let us prove this formula by recursion. Case $r=1$ is
  just~\eqref{trieq}. Let us see precisely what happens for $r=2$. We
  start from
  \begin{equation*}
    V_s(t)=-\sum_{\substack{d_1|\exp G,\\ d_1\neq1}}
    \frac{N_{d_1}}{d_1} V_s(d_1t)+M_1^{-1}\Gamma(t)
  \end{equation*}
  which we rewrite as
  \begin{equation*}
    V_s(t)=-\sum_{\substack{d_1|\exp G,\\ d_1\neq1,\\ d_1\le z}}
    \frac{N_{d_1}}{d_1} V_s(d_1t)
    -\sum_{\substack{d_1|\exp G,\\ d_1\neq1,\\ d_1> z}}
    \frac{N_{d_1}}{d_1} V_s(d_1t)
    +M_1^{-1}\Gamma(t).
  \end{equation*}
  We use again this equation on $V_s(d_1t)$ when $d_1\le z$, and $z/d_1$
  rather than $z$, getting
  \begin{multline*}
    V_s(t)=\sum_{\substack{d_1|\exp G,\\ d_1\neq1,\\ d_1\le z}}
    \sum_{\substack{d_2|\exp G,\\ d_2\neq1,\\ d_1d_2\le z}}
    \frac{N_{d_1}}{d_1}\frac{N_{d_2}}{d_2} V_s(d_1d_2t)
    \\+
    \sum_{\substack{d_1|\exp G,\\ d_1\neq1,\\ d_1\le z}}
    \sum_{\substack{d_2|\exp G,\\ d_2\neq1,\\ d_1d_2> z}}
    \frac{N_{d_1}}{d_1}\frac{N_{d_2}}{d_2} V_s(d_1d_2t)
    -\sum_{\substack{d_1|\exp G,\\ d_1\neq1,\\ d_1> z}}
    \frac{N_{d_1}}{d_1} V_s(d_1t)
    \\
    -\sum_{\substack{d_1|\exp G,\\ d_1\neq1,\\ d_1\le z}}
    \frac{N_{d_1}}{d_1} M_1^{-1}\Gamma(t)
    +M_1^{-1}\Gamma(t).
  \end{multline*}
  To go from $r$ to $r+1$, we select the divisors
  $d_r$ that are such that $d_1d_2\cdots d_r\le z$
  and employ~\eqref{trieq} on $V_s(d_1\cdots d_rt)$. 
\end{proof}

\begin{lem}
  The coefficients of a product $N_{d_1}N_{d_2}\cdots N_{d_v}$ are at most (in
  absolute value) equal
  to~$|G^\sharp|^{v-1}$, where $G^\sharp$ is the set of
  \emph{lattice-invariant} classes (which is also the number of
  cyclic subgroups of $G$).
\end{lem}

\subsubsection*{End of the proof of Theorem~\ref{precise}}

The formula~\eqref{recursion} with $t=1$ contains most of our
proof. We only have to control the error term, which is our next task.

The number of possible $d$'s is at most the number of divisors of
$\exp G$ minus~1, so at most $d(\exp G)$. The coefficients of a typical product $N_{d_1}\cdots
N_{d_v}$ are of size at most $|G^\sharp|^{v-1}$, we divide each
coefficient by $d_1\cdots d_v$ which is at least $z$, and we have at
most $d(\exp G)^{v}$ $v$-tuples $(d_1,\ldots,d_v)$. As a
consequence, each coordinate, says $y$, of the vector
\begin{equation*}
  \sum_{1\le v\le r-1}(-1)^v
  \sum_{\substack{d_1\cdots d_{v-1}\le z,\\
      d_1\cdots d_{v-1}d_v>z}}
  \frac{N_{d_1}}{d_1}\ldots
  \frac{N_{d_v}}{d_v} V_s(d_1\ldots d_v t)
\end{equation*}
satisfies
\begin{equation*}
  |y|\le (r-1)\,\frac{\bigl(|G^\sharp|d(\exp G)\bigr)^{r-1}}{z|G^\sharp|}
  \max_{D\ge 2^r}\|V_s(Dt)\|.
\end{equation*}
We deal similarly with the coordinates of the vector
\begin{equation*}
  (-1)^{r}
  \sum_{\substack{d_1\cdots d_{r}\le z}}
  \frac{N_{d_1}}{d_1}\ldots
  \frac{N_{d_r}}{d_r} V_s(d_1\ldots d_r t)
\end{equation*}
except that the denominator $d_1\cdots d_r$ is not especially larger
than $z$; we however select $z=2^r$ to ensure this condition. This
means that only $d_1=d_2=\ldots=d_r=2$ is admissible.
So, on combining both, we see that
\begin{multline}
  \biggl\|  V_s(1)-\sum_{0\le v\le r-1}(-1)^v
  \sum_{\substack{d_1\cdots d_{v}\le 2^r}}
  \frac{N_{d_1}}{d_1}\ldots
  \frac{N_{d_{v}}}{d_{v}}M_1^{-1}\Gamma_s(d_1\ldots d_v)
  \biggr\|
  \\\le
  \frac{1}{2}\biggl(1+\frac{r-1}{|G^\sharp|}\biggr)\,\biggl(\frac{|G^\sharp|d(\exp G)}{2}\biggr)^{r-1}
  \max_{D\ge 2^r}\|V_s(D)\|.
\end{multline}
To complete the proof, we simply need a bound for $\max_{D\ge
  2^r}\|V_s(D)\|$ and such a bound is provided by the next lemma.

\begin{lem}
  Let $\mathcal{A}$ be a subset of the $G=(\Z{q})^\times$. Let $f>1$
  be a real parameter. We have
  \begin{equation*}
    \bigl|\log \zeta_P(f;q,\mathcal{A})\bigr|\le
    \frac{1+P/(f-1)}{P^{f}}.
  \end{equation*}
\end{lem}

\begin{proof}
  We use
  \begin{equation*}
    \log \zeta_P(f;q,\mathcal{A})
    =-\sum_{\substack{p\in\mathcal{A}, \\ p\ge
        P}}\sum_{k\ge1}\frac{1}{k p^{kf}}
  \end{equation*}
  hence, by using a comparison to an integral, we find that
  \begin{equation*}
     \Bigl|\log \zeta_P(f;q,\mathcal{A})\Bigr|
     \le
     \sum_{n\ge P}\frac{1}{n^f}\le \frac{1}{P^f}+\int_{P}^\infty\frac{dt}{t^f}
  \end{equation*}
\end{proof}

\section{A detailed example modulo 7}
\label{mod7}

Wet set $G=(\Z{7})^\times$. We find that 
\begin{equation*}
  \mathscr{G}=\bigl\{\{1\}, \{1, 6\}, \{1, 2, 4\}, \{1, 2, 3, 4, 5, 6\}\bigr\}
\end{equation*}
(indexed in this order) and that
\begin{equation*}
  G^\sharp=\bigl\{\{1\}, \{6\}, \{2, 4\}, \{3, 5\}\bigr\},
\end{equation*}
also indexed in that order. There are 6 Dirichlet characters whose
values are given by (with $\zeta_6=\exp(2i\pi/6)$)
\begin{center}
  \begin{tabular}{|c|c|c|c|c|c|c|}
  \hline
  &1&2&3&4&5&6\\
  \hline
  $\chi_0$&1&1&1&1&1&1\\
  \hline
  $\chi_1$&1&$\zeta_6^2$&$\zeta_6$&$-\zeta_6$&$-\zeta_6^2$&$-1$\\
  \hline
  $\chi_2$&1&$-\zeta_6$&$\zeta_6^2$&$\zeta_6^2$&$-\zeta_6$&1\\
  \hline
  $\chi_3$&1&1&$-1$&1&$-1$&$-1$\\
  \hline
  $\chi_4$&1&$\zeta_6^2$&$-\zeta_6$&$-\zeta_6$&$\zeta_6^2$&1\\
  \hline
  $\chi_5$&1&$-\zeta_6$&$-\zeta_6^2$&$\zeta_6^2$&$\zeta_6$&$-1$\\
  \hline
\end{tabular}
\end{center}
We obtain this list with the command
\begin{equation*}
  \texttt{[[e(n) for n in xrange(1,7)] for e in GetStructure(7)[5]]}
\end{equation*}
and the remark $\zeta_6-1=\zeta_6^2$. The 8th component of
\texttt{GetStructure(7)} gives the index of the characters that are
trivial on the above subgroups, its value is thus
\begin{equation*}
  [[0,1,2,3,4,5], [0, 2, 4], [0, 3], [0]].
\end{equation*}
The vector $\Gamma_s(t)$ is given by (it is defined by~\eqref{defGammaoft})
\begin{equation*}
  \Gamma_s(t) =\left|\begin{array}{l}
                       \log \prod_{0\le i\le 5}L_P(ts,\chi_i)\\
                       \log \prod_{i\in \{0,2,4\}}L_P(ts,\chi_i)\\
                       \log (L_P(ts,\chi_0)L_P(ts,\chi_3))\\
                       \log L_P(ts,\chi_0)
                     \end{array}
                   \right.
\end{equation*}
while
\begin{equation*}
  V_s(t) =\left|\begin{array}{l}
                       -\log \prod_{\substack{p\equiv 1[7],\\ p\ge P}}(1-1/p^{ts})\\
                       -\log \prod_{\substack{p\equiv 6[7],\\ p\ge P}}(1-1/p^{ts})\\
                       -\log \prod_{\substack{p\equiv 2,4[7],\\ p\ge P}}(1-1/p^{ts})\\
                       -\log \prod_{\substack{p\equiv 3,5[7],\\ p\ge P}}(1-1/p^{ts})
                     \end{array}
                   \right.
                 \end{equation*}
Now that the players and the surrounding environment has been
described, let us turn towards the main step of our proof: the
recursion~\eqref{initeq}. We first check that
\begin{align*}
  \gamma(\{1\},t)
  &=6x^{t}_{\{1\}}+3x_{\{6\}}^{2t}+2x_{\{2,4\}}^{3t}+x_{\{3,5\}}^{6t},
  \\
  \gamma(\{1,6\},t)
  &=3x^{t}_{\{1\}}+3x_{\{6\}}^{t}+x_{\{2,4\}}^{3t}+x_{\{3,5\}}^{3t},
  \\
  \gamma(\{1,2,4\},t)
  &=2x^{t}_{\{1\}}+x_{\{6\}}^{2t}+2x_{\{2,4\}}^{t}+x_{\{3,5\}}^{2t},
  \\
  \gamma(\{1,2,3,4,5,6\},t)
  &=x^{t}_{\{1\}}+x_{\{6\}}^{t}+x_{\{2,4\}}^{t}+x_{\{3,5\}}^{t}.
\end{align*}
Whence the relation
\begin{equation*}
  \Gamma_s(t)
  =
  M_1
  V_s(t)
  +
  M_2
  V_s(2t)
  +
 M_3
  V_s(3t)
  +
  M_6
  V_s(6t)
\end{equation*}
with
\begin{align*}
  M_1&=\begin{pmatrix}
    6&0&0&0\\
    3&3&0&0\\
    2&0&2&0\\
    1&1&1&1
  \end{pmatrix}
           ,\quad
  M_2=\begin{pmatrix}
    0&3&0&0\\
    0&0&0&0\\
    0&1&0&1\\
    0&0&0&0
  \end{pmatrix}
  \\  
  M_3&= \begin{pmatrix}
    0&0&2&0\\
    0&0&1&1\\
    0&0&0&0\\
    0&0&0&0
  \end{pmatrix}
           ,\quad
  M_6=\begin{pmatrix}
    0&0&0&1\\
    0&0&0&0\\
    0&0&0&0\\
    0&0&0&0
  \end{pmatrix}.
\end{align*}
The call \texttt{GetM1Inverse(7,GetStructure(7))\^\,(-1)} produces the
matrix $M_1$. The matrices $N_d=dM_1^{-1}M_d$ are obtained by
\texttt{GetNds(7,GetStructure(7))}. They are
\begin{equation*}
  N_2=\begin{pmatrix}
    0&1&0&0\\
    0&-1&0&0\\
    0&-1&0&1\\
    0&1&0&-1
  \end{pmatrix}
  ,\quad 
  N_3= \begin{pmatrix}
    0&0&1&0\\
    0&0&-1&1\\
    0&0&-1&0\\
    0&0&1&-1
  \end{pmatrix}
           ,\quad
  N_6=\begin{pmatrix}
    0&0&1&-1\\
    0&0&0&-1\\
    0&0&0&-1\\
    0&0&0&1
  \end{pmatrix}.
\end{equation*}
In order to check our script, we mention that the call
\begin{center}
  \texttt{GetM1Inverse(7,GetStructure(7))\^\,(-1)*GetNd(2,7,GetStructure(7))/2}
\end{center}
gives $M_2$ for instance (and one can replace the parameter 2 that
occurs twice with 3 or~6 to get $M_3$ and $M_6$).
We have reached
\begin{equation*}
  V_s(t)=
  M_1^{-1}\Gamma_s(t)
  -
  \frac{N_2}{2}
  V_s(2t)
  -\frac{N_3}{3}
  V_s(3t)
  -\frac{N_6}{6}
  V_s(6t).
\end{equation*}
Our objective is $V_s(1)$ and we know how to compute $\Gamma_s(t)$
while,  when $d$ is large,
$V_s(dt)$ vanishes approximately; it is thus enough
to iterate the above formula. We end the numerical example here.

\section{Rational Euler Products}

Let us recall the Witt decomposition.
The readers will find in \cite[Lemma 1]{Moree*00} a result of the same
flavour. We have simply modified the proof and setting as to
accomodate polynomials having real numbers for coefficients. 
\begin{lem}
\label{Wittpoly}
Let $F(t) = 1+a_1 t+\ldots+a_{\delta}t^{\delta} \in \mathbb{R}[t]$ be a 
polynomial of degree~$\delta$. 
Let $\alpha_{1},\ldots,\alpha_{\delta}$ be the inverses of its
roots. Put $s_{F}(k) =\alpha_{1}^{k}+\ldots+\alpha_{\delta}^{k}$. The
$s_{F}(k)$ are integers and satisfy the Newton-Girard recursion
\begin{equation}
\label{recursion}
s_{F}(k)+a_1s_F(k-1)+\ldots+a_{k-1}s_{F}(1)+ka_{k}=0,   
\end{equation}
where we have defined $a_{\delta+1} =a_{\delta+2}=\ldots=0$. Put 
\begin{equation}
    \label{bfk}
b_{F}(k)=\frac{1}{k}\sum_{d|k}\mu({k}/{d})s_{F}(d).
\end{equation}
Let $\beta\ge1$ be such that $\beta\ge\max_j|1/|\alpha_j|$. When $t$
belongs to any segment $\subset (-\beta,\beta)$, we have
\begin{equation}
\label{Fhatb}
F(t)=\prod_{j=1}^{\infty}(1-t^{j})^{b_{F}(j)}
\end{equation}
where the convergence is uniform in the given segment.
\end{lem}
And how does the mathematician E. Witt
enter the scene? In the paper
\cite{Witt*37} on Lie algebras, Witt produced in equation $(11)$
therein a decomposition that is the prototype of the above expansion.

\begin{proof}
  Since we follow the proof of \cite[Lemma 1]{Moree*00}, we shall be
  rather sketchy. We write
  $F(t)=\prod_{i}(1-\alpha_it)$. We thus have
  \begin{equation*}
    \frac{tF'(t)}{F(t)}=\sum_i\frac{\alpha_i t}{1-\alpha_it}
    =\sum_{k\ge1}s_F(k)t^k.
  \end{equation*}
  This series is absolutely convergent in any disc $|t|\le b<1/\beta$
  where $\beta=\max_j(1/|\alpha_j|)$. 
  We may also decompose $tF'(t)/F(t)$ in Lambert series as
  \begin{equation*}
    \frac{tF'(t)}{F(t)}=\sum_{j\ge1}b_F(j)\frac{jt^j}{1-t^j}
  \end{equation*}
  as some series shuffling in any disc of radius $b<\min(1,1/\beta)$ shows.
  The lemma follows readily by integrating the above relation.
\end{proof}

\begin{lem}\label{apriorimaj}
  We use the hypotheses and notation of Lemma~\ref{Wittpoly}. Let $\beta\ge2$ be larger than the inverse of the modulus of all the roots of $F(t)$. We have
  \begin{equation*}
      |b_F(k)|\le2\deg F \cdot \beta^k/k.
  \end{equation*}
\end{lem}

\begin{proof}
  We clearly have $|s_F(j)|\le \deg F\cdot \beta^j,$ so that
 \begin{align*}
   |b_F(k)|
   &\le \frac{\deg F}{k}\sum_{1\le j\le k}\beta^j
   \le \frac{\deg F}{k}\beta\frac{\beta^k-1}{\beta-1}
   \\&\le\frac{\deg F}{k}\frac{\beta^k}{1-1/\beta}
   \le 2\deg F\cdot \beta^k/k.
 \end{align*}
\end{proof}
There are numerous easy upper estimates for the inverse of the modulus of all the roots of $F(t)$ in terms of its coefficients. Here is a simplistic one.
\begin{lem}
\label{easybeta}
Let $F(X)=1+a_1X+\ldots+a_\delta X^\delta$ be a polynomial of
degree~$\delta$. Let $\rho$ be one of its roots. Show that, either $|\rho|\ge 1$ or
$1/|\rho|\le |a_1|+|a_2|+\ldots+|a_\delta|$.
\end{lem}

\begin{proof}
  The readers may first notice that
\begin{equation*}
   (1/\rho)^\delta =
    -a_1(1/\rho)^{\delta-1}
    -a_2(1/\rho)^{\delta-2}-\ldots -a_\delta.
  \end{equation*}
  The conclusion is easy.
\end{proof}
\begin{proof}[Proof of Theorem~\ref{PM1}]
The proof requires several steps. The very first one is a direct consequence of~\eqref{Fhatb}, which leads to the identity 
\begin{equation}
\label{formal-FG}
\frac{F(t)}{G(t)}=\prod_{j=2}^\infty(1-t^j)^{b_F(j)-b_G(j)}.
\end{equation}
The absence of the $j=1$ term is due to our assumption that $(F(X)- G(X))/X^2\in \mathbb{Z}[X]$.   Up to this point \eqref{formal-FG} is only established as a formal identity. Our second step is to establish~\eqref{formal-FG} for all $t\in\mathbb{C}$ with $|t|<1/\beta$ and to control the rate of convergence. 
By Lemma~\ref{apriorimaj}, we know that
$
    |b_F(j)-b_G(j)|\le 2\max(\deg F,\deg G)\beta^j/j
$.
Therefore, for any bound $J$, we have
\begin{equation}
   \label{tailJ}
    \sum_{j\ge J+1}|t^j||b_F(j)-b_G(j)|
    \le
    2\max(\deg F,\deg G)\frac{|t\beta|^{J+1}}{(1-|t\beta|)(J+1)},
\end{equation}
as soon as $|t|<1/\beta$. We thus have
\begin{equation}
\label{true-FG}
\frac{F(t)}{G(t)}=\prod_{2\le j\le J}(1-t^j)^{b_F(j)-b_G(j)}\times I_1,
\end{equation}
where
$|\log I_1|\le2\max(\deg F,\deg G)|t\beta|^{J+1}/[(1-|t\beta|)(J+1)]$.

Now that we have the expansion~\eqref{true-FG} for each prime $p$, we may combine them. We readily get
\begin{equation*}
    \prod_{\substack{p\ge P,\\ p\in\mathcal{A}}}\frac{F(1/p)}{G(1/p)}
    =
    \prod_{\substack{p\ge P,\\ p\in\mathcal{A}}}\prod_{2\le j\le J}(1-p^{-j})^{b_G(j)-b_F(j)}\times I_2,
\end{equation*}
where $I_2$ satisfies
\begin{align*}
    |\log I_2|
    &\le
    2\max(\deg F,\deg G)\sum_{p\ge P}\frac{\beta^{J+1}}{1-\beta/P}\frac{1}{(J+1)p^{J+1}}
    \\&\le
    \frac{2\max(\deg F,\deg G)\beta^{J+1}}{(1-\beta/P)(J+1)}\biggl(
    \frac{1}{P^{J+1}}+\int_{P}^{\infty}\frac{dt}{t^{J+1}}\biggr)
    \\&\le
    \frac{2\max(\deg F,\deg G)(\beta/P)^J\beta}{(1-\beta/P)(J+1)}\biggl(\frac{1}{P}+\frac{1}{J}\biggr),
\end{align*}
since $P\ge2$ and $J\ge3$.
As announced earlier, we may rearrange the product over the primes~$p$ and get
\begin{equation*}
    \prod_{\substack{p\ge P,\\ p\in\mathcal{A}}}\frac{F(1/p)}{G(1/p)}
    =
    \prod_{2\le j\le J}\zeta_{P}(j;q,\mathcal{A})^{b_G(j)-b_F(j)}\times 
    I_2.
  \end{equation*}
  The last step is to replace $\zeta_{P}(j;q,\mathcal{A})$ by the
  approximation, say $\zeta_{P}(j;q,\mathcal{A}|r)$ given by~\eqref{defzetaPsqAr}.
We find that
\begin{equation*}
  \prod_{\substack{p\ge P,\\ p\in\mathcal{A}}}\frac{F(1/p)}{G(1/p)}
    =
    \prod_{2\le j\le J}\zeta_P(j;q,\mathcal{A}|r)^{b_F(j)-b_G(j)}\times I_3,
\end{equation*}
where $I_3$ satisfies
\begin{align*}
    |\log I_3|
  &\le
   C\sum_{2\le j\le J}|b_F(j)-b_G(j)|
  \frac{1+P/(2^rj-1)}{P^{j2^r}}
    +|\log I_2|
    \\&\le
    C\sum_{2\le j\le
  J}2\max(\deg F,\deg G)\frac{\beta^j}{j}
  \frac{1+2^{-r}P}{P^{j2^r}}
  +|\log I_2|.
\end{align*}
with
\begin{equation*}
  C= \frac{1}{2}\biggl(1+\frac{r-1}{|G^\sharp|}\biggr)
    \,\biggl(\frac{|G^\sharp|d(\exp G)}{2}\biggr)^{r-1}.
\end{equation*}
Therefore (and since $r\ge2$)
\begin{multline}
  \frac{|\log I_3|}{2\max(\deg F,\deg G)}\label{precbound}
  \le
   \frac{1}{4}\biggl(1+\frac{r-1}{|G^\sharp|}\biggr)
  \,\biggl(\frac{|G^\sharp|d(\exp G)}{2}\biggr)^{r-1}
  \frac{\beta^2}{P^{2^{r+1}}}
  \frac{1+2^{-r}P}{1-\beta/P^{4}}
  \\+  
  \frac{(\beta/P)^J\beta}{(1-\beta/P)(J+1)}\biggl(\frac{1}{P}+\frac{1}{J}\biggr)
\end{multline}
and this ends the proof.
\end{proof}

\section{Counting the number of Lattice-Invariant Classes}
It is of interest to count how many lattice-invariant classes there
are, i.e. to determine the cardinality of $G^\sharp$ which is equally
the number of cyclic subgroups, i.e. the cardinality of
$\mathscr{G}$. We proceed in several steps.
\begin{lem}
  \label{count1}
  Let $d\ge1$ and $q\ge1$ be two integers. The number $\rho(q;d)$ of
  solutions to the equation $x^d\equiv1[q]$ is a multiplicative
  function of the
  variable~$q$. When $p$ is a prime, we find that
  \begin{equation*}
    \rho(p^\alpha;d)=
    \begin{cases}
      (d,p^{\alpha-1}(p-1))&\text{if $p\neq2$,}\\
      1&\text{if $p=2$ and $\alpha=1$,}\\
      1&\text{if $p=2$, $\alpha\ge2$ and $d$ odd,}\\
      2(d,2^{\alpha-2})&\text{if $p=2$, $\alpha\ge2$ and $d$ even.}
    \end{cases}
  \end{equation*}
  The function $d\mapsto \rho(q;d)$ is also multiplicative.
\end{lem}

\begin{proof}
  The multiplicative character of $\rho(q;d)$ stems from the Chinese
  Remainder Theorem. In $\mathbb{Z}/p^\alpha\mathbb{Z}$ and $p\neq2$, the equation
  $x^d\equiv1[p^\alpha]$ has $(d,p^{\alpha-1}(p-1))$ as stated in
  \cite[Corollary 2.42]{Niven-Zuckerman-Montgonemry*91}; it is an easy
  consequence of the fact that
  $(\mathbb{Z}/p^\alpha\mathbb{Z})^\times$ is cyclic in this case.

  When $p=2$, the equation $x^d\equiv1[2]$ has exactly one solution, namely
  $x=1$. When $p=2$ and $\alpha\ge2$, the multiplicative group
  $(\mathbb{Z}/p^\alpha\mathbb{Z})^\times$ is isomorphic to the direct product
  $(\mathbb{Z}/2\mathbb{Z},+)\times
  (\mathbb{Z}/2^{\alpha-2}\mathbb{Z},+)$. We find as a consequence of \cite[Proposition
  4.2.2]{Ireland-Rosen*90} that $\rho(2^\alpha,d)=2(d,2^{\alpha-2})$.
  
  The multiplicativity of the function $d\mapsto \rho(q;d)$ folows
  from the explicit expression of $\rho(q;d)$: it is a product (over
  prime factors of $q$) of multiplicative functions of the variable~$d$.
\end{proof}

\begin{lem}
  \label{count2}
  The number $\rho^*(q;d)$ of elements of order $d$ in
  $(\mathbb{Z}/q\mathbb{Z})^\times$ is given by 
  \begin{equation*}
    \sum_{\ell|d}\mu(d/\ell)\rho(q;\ell)
  \end{equation*}
  where $\rho(q;d)$ is defined and determined in Lemma~\ref{count1}.
\end{lem}

\begin{proof}
  This is a consequence of the Moebius inversion formula as, by
  classifying the solution of $x^d\equiv 1[q]$ by their order, we find that
  $\rho(q;d)=\sum_{\ell|d}\rho^*(q;\ell)$. 
\end{proof}

\begin{prop}
  \label{count3}
  When $q\ge3$, the number $|\mathscr{G}|$ of cyclic subgroups of
  $(\mathbb{Z}/q\mathbb{Z})^\times$ is given by
  \begin{equation*}
    |\mathscr{G}|=
    \prod_{\substack{p|\varphi(q),\\p\neq 2}}\frac{p-2}{p-1}
    \sum_{\substack{d|\varphi(q),\\ 2|d}}\frac{\rho(q;d)}{\varphi(d)}
    \prod_{\substack{p|d,\\ p|\varphi(q)/d,\\ p\neq2}}\frac{(p-1)^2}{p(p-2)}
    \prod_{\substack{p|d,\\ p\nmid\varphi(q)/d,\\
        p\neq2}}\frac{p-1}{p-2}
    \prod_{\substack{2|d,\\ 2|\varphi(q)/d}}\frac{1}{2}
  \end{equation*}
  where $\rho(q;d)$ is defined and determined in Lemma~\ref{count1}.
\end{prop}
We have checked this expression with Sage via the function
\texttt{CardClassList} of our script. The values have been checked
against a direct count: we have the list of
lattice-invariant classes, hence their number.

\begin{proof}
  Each cyclic subgroup of order $d$ has $\varphi(d)$
  generators. Hence the number of cyclic subgroups
  of order $d$ is equal to $\rho^*(q;d)/\varphi(d)$, whence, by Lemma~\ref{count2}, 
  \begin{align*}
    |\mathscr{G}|
    &=
      \sum_{d|\varphi(q)}\frac{1}{\varphi(d)} \sum_{\ell|d}\mu(d/\ell)\rho(q;\ell)
    \\&=
    \sum_{\ell|\varphi(q)}\rho(q;\ell)
    \sum_{\ell|d|\varphi(q)}\frac{\mu(d/\ell)}{\varphi(d)}.
  \end{align*}
  To evaluate the inner sum, write $\varphi(q)=h_1h_2h_3$, where $h_1$
  is the product of the $p^{v_p(\varphi(q))}$ with $p|\ell$ and
  $p|\varphi(q)/\ell$, then $h_2$ is the product of the
  $p^{v_p(\varphi(q))}$ with $p|\ell$ but $p\nmid \varphi(q)/\ell$ and 
  $(h_3,\ell)=1$ is what remains after division by $h_1h_2$. We
  readily find that
  \begin{equation*}
    \sum_{\ell|d|\varphi(q)}\frac{\mu(d/\ell)}{\varphi(d)}
    =\frac{1}{\varphi(\ell)}
    \prod_{p|h_1}\Bigl(1-\frac{1}{p}\Bigr)
    \prod_{p|h_2}1
    \prod_{p|h_3}\Bigl(1-\frac{1}{p-1}\Bigr).
  \end{equation*}
  This vanishes when $2|h_3$, so we can restrict our attention to even
  $\ell$'s. In which case we get
  \begin{align*}
    \sum_{\ell|d|\varphi(q)}\frac{\mu(d/\ell)}{\varphi(d)}
    &=\frac{1}{\varphi(\ell)}
    \prod_{\substack{p|\varphi(q),\\p\neq2}}\frac{p-2}{p-1}
    \prod_{\substack{p|h_1,\\ p\neq 2}}\frac{(p-1)^2}{p(p-2)}
    \prod_{\substack{p|h_2,\\ p\neq 2}}\frac{p-1}{p-2}
    \times
    \biggl(\frac12\text{when $2^{v_2(\varphi(q))}\nmid \ell$}
    \biggr).
  \end{align*}
  We reverse to the variable $d$ rather than $\ell$ to write our lemma.
  We have also used the condition $q\ge3$ to ensure that $2|\varphi(q)$.
\end{proof}

\section{Notes on the implementation}

The parameter $r$ is not very large, typically between~2 and~8. Since
in~\eqref{fineq}, several products $d=d_1\cdots d_v$ are equal, we store
the computed values of $\Gamma_s(dt)$ in the
dictionary~\texttt{ComputedGammas} in the function \texttt{GetVs} of
the script \texttt{LIEP.sage}. We proceed similarly with the dictionary
\texttt{ComputedProductNdsM1Inverse} for the products
$N_{d_1}\cdots N_{d_v}M_1^{-1}$ in . Since the \emph{list}
$[d_1,\cdots,d_v]$ cannot be a key for such a dictionary, we simply
replace it by the \emph{tuple}  $(d_1,\cdots,d_v)$.


Concerning the general structure, the function \texttt{GetStructure}
computes all the algebraical quantities that we need: the list of
cyclic subgroups, the one of lattice-invariant classes, the exponent
of our group, its character group, the set of invertible classes and,
for each cyclic subgroup, the set of characters that are trivial on it.

Once the script is loaded via
\texttt{load('LIEP.sage')}, a typical call will be
\begin{center}
  GetVs(12, 2, 100, 300)
\end{center}
to compute modulo~12 the possible constants with $s=2$, asking for 100
decimal digits and using $P=300$. The output is self
explanatory. The number of decimal digits asked for is roughly handled
and one may lose precision in between, but this is
indicated at the end (we observed no such phenomenon, but it may still
happen!). A more precise treatment would first check the
output and if the precision attained would not be enough, increase
automatically this parameter. We prefer to let the users do that by
themselves. The digits presented when \texttt{WithLaTeX} $=1$ are always
accurate. Note that we expect the final result to be of 
size roughly unity, so we ask for is not the relative precision but
the number of decimals. Hence, in the function \texttt{GetGamma}, we
replace by an approximation of~0 the values that we know are
insignificantly small. This is a true time-saver.

There are two subsequent optional parameters
\texttt{Verbose} and \texttt{WithLaTeX}. The first one may take the
values 0, 1 and 2; when equal to 0, the function will simply do its
job and return the list of the invariant classes and the one of the
computed lower and upper values. When equal to 1, its default
value, some information on the computation is given. At level 2, more
informations is given, but that should not concern the casual user.
When the parameter \texttt{Verbose} is at least 1 and
\texttt{WithLaTeX} is 1, the values of the constants will be further
presented in a format suitable for inclusion in a \LaTeX -file.
For instance, the call
\begin{center}
  GetVs(12, 2, 100, 100, 1, 1)
\end{center}
is the one used to prepare this document.

To compute the Euler products as explained in Theorem~\ref{PM1}, we
have the function \texttt{GetEulerProds(q, F, G, nbdecimals, bigP =
  100, Verbose = 1, WithLaTeX = 0)}. Note that the parameter
\texttt{bigP} may be increased during the run of the program to ensure
that $P\ge2\beta$ (a condition that is most of the time satisfied). We
reused the same structure as the function \texttt{GetVs}, without
calling it: this is to also keep all the precomputed datas. Since the
coefficients $|b_F(j)-b_G(j)|$ may increase like $\beta^j$, we
increase the working precision by~$J\log\beta /\log 2$.


\subsection*{Checking}

The values given here have been checked in several manners. The
co-authors of this paper have computed several of the next values via
independent scripts. We also provide the function
\texttt{GetVsChecker(q, s, borne = 10000)} which computes approximate
values of the same Euler products by simply truncating the Euler
product representation.
We checked with positive result the stability of our results with
respect of the variation of the parameter~$P$.
This proved to be a very discriminating test.

Furthermore, approximate values for Shank's
and Lal's constants are known (Finch in \cite{Finch*03} gives 10
digits) and we agree on those. Finally,
the web site \cite{Gourdon-Sebah*10} by
X. Gourdon and P. Sebah is nowadays difficult to decypher but a postscript
version is available on the same page. They give in section 4.4 the first fifty
digits of the constant they call~$A$ and which is
\begin{align*}
  \label{eq:7}
  \smash{\frac{\pi^2}{2}\prod_{p\equiv1[4]}\biggl(1-\frac{4}{p}\biggr)\biggl(\frac{p+1}{p-1}\biggr)^2}
  = 1.&95049\,11124\,46287\,07444\,65855\,65809\,55369
        \\&25267
  \,08497\,71894\,30550\,80726\,33188\,94627
  \\&61381\,60369
  \,39924\,26646\,98594\,38665\cdots
\end{align*}
Our result match the one of \cite{Gourdon-Sebah*10}.

\section{Some results}
In this part, we exhibit some results for $s=2$ and
small $q$'s. We decided to produce 100 decimal digits each time. Each
computation took at most five seconds and we selected uniformly
$P=100$.

\subsection*{Modulo 3}
{\footnotesize
  \begin{align*}
    \smash{\prod_{p\equiv1[3]}(1-p^{-2})^{-1}}
    =
    1.&03401\,48754\,14341\,88053\,90306\,44413\,04762\,85789\,65428\,48909
        \\&98864\,16825\,03842\,12222\,45871\,09635\,80496\,21707\,98262\,05962\cdots
  \end{align*}
  \vspace*{-5mm}
  \begin{align*}
    \smash{\prod_{p\equiv2[3]}(1-p^{-2})^{-1}}
    =1.&41406\,43908\,92147\,63756\,55018\,19079\,82937\,99076\,95069\,39316
         \\&21750\,39924\,96242\,39281\,06992\,08849\,94537\,54858\,50247\,51141\cdots
  \end{align*}
  }

\subsection*{Modulo 4}
{\footnotesize
  \begin{align*}
    \smash{\prod_{p\equiv1[4]}(1-p^{-2})^{-1}}
    =
    1.&05618\,21217\,26816\,14173\,79307\,65316\,21989\,05875\,80425\,46070
\\&80120\,04306\,19830\,27928\,16062\,22693\,04895\,12958\,37291\,59718\cdots
  \end{align*}
  \vspace*{-5mm}
  \begin{align*}
    \smash{\prod_{p\equiv3[4]}(1-p^{-2})^{-1}}
    =1.&16807\,55854\,10514\,28866\,96967\,37064\,04040\,13646\,79021\,45554
\\&79928\,40563\,68111\,38106\,59377\,71094\,66904\,07472\,79588\,48702\cdots
  \end{align*}
  }

\subsection*{Modulo 5}
{\footnotesize
  \begin{align*}
    \smash{\prod_{p\equiv1[5]}(1-p^{-2})^{-1}}
    =
    1.&01091\,51606\,01019\,52260\,49565\,84289\,51492\,09845\,38627\,58173
\\&85237\,32024\,20089\,25161\,37424\,56726\,37093\,96197\,69455\,89218\cdots
  \end{align*}
  \vspace*{-5mm}\begin{align*}
    \smash{\prod_{p\equiv2,3[5]}(1-p^{-2})^{-1}}
    =1.&55437\,60727\,20889\,22081\,75902\,82565\,55177\,56056\,30147\,34257
    \\&40072\,50077\,94457\,39239\,00871\,38641\,44091\,80733\,87878\,70683\cdots
  \end{align*}
  \vspace*{-5mm}\begin{align*}
    \smash{\prod_{p\equiv4[5]}(1-p^{-2})^{-1}}
    =1.&00496\,03239\,22297\,55899\,37496\,24810\,25218\,47955\,10294\,18802
    \\&28801\,99528\,37852\,15071\,27700\,70076\,98854\,32491\,36118\,00619\cdots
  \end{align*}
  }
\subsection*{Modulo 7}
{\footnotesize
  \begin{align*}
    \smash{\prod_{p\equiv1[7]}(1-p^{-2})^{-1}}
    =   1.&00222\,95338\,19740\,42627\,18641\,59138\,22019\,24486\,37565\,40128
    \\&87922\,82973\,79678\,21741\,90308\,08041\,42707\,36575\,28295\,76151\cdots
  \end{align*}
  \vspace*{-5mm}\begin{align*}
    \smash{\prod_{p\equiv2,4[7]}(1-p^{-2})^{-1}}
    =1.&34984\,62543\,65273\,20787\,74772\,44978\,62277\,76508\,69021\,24860
    \\&12031\,69999\,35719\,21654\,93824\,75777\,02051\,36300\,53459\,76601\cdots
  \end{align*}
  \vspace*{-5mm}\begin{align*}
    \smash{\prod_{p\equiv3,5[7]}(1-p^{-2})^{-1}}
    =1.&18274\,26007\,67364\,09208\,00286\,83933\,15918\,51718\,05360\,46335
    \\&82633\,06344\,66854\,90324\,90537\,21799\,81486\,90001\,86365\,91391\cdots
  \end{align*}
  \vspace*{-5mm}\begin{align*}
    \smash{\prod_{p\equiv6[7]}(1-p^{-2})^{-1}}
    =1.&00705\,20326\,03074\,04805\,67193\,52428\,88870\,69289\,36714\,73687
    \\&58335\,65893\,11634\,74829\,60947\,12069\,41243\,26265\,99553\,53536\cdots
  \end{align*}
  }
\subsection*{Modulo 8}
{\footnotesize
  \begin{align*}
    \smash{\prod_{p\equiv1[8]}(1-p^{-2})^{-1}}
    =1.&00483\,50650\,34191\,18711\,83598\,31169\,10411\,95979\,07317\,54340
    \\&88789\,55156\,06711\,74639\,62051\,31056\,35207\,32105\,88068\,58783\cdots
  \end{align*}
  \vspace*{-5mm}\begin{align*}
    \smash{\prod_{p\equiv3[8]}(1-p^{-2})^{-1}}
    =1.&13941\,87771\,08211\,51502\,70589\,30773\,34020\,88725\,59961\,09629
    \\&48302\,25821\,27411\,02101\,65577\,60742\,91446\,59374\,91512\,33349\cdots
  \end{align*}
  \vspace*{-5mm}\begin{align*}
    \smash{\prod_{p\equiv5[8]}(1-p^{-2})^{-1}}
    =1.&05109\,99849\,42183\,30793\,68775\,56006\,33505\,68012\,01018\,45817
    \\&85080\,59912\,94207\,39729\,30485\,58783\,38889\,50479\,59255\,34495\cdots
  \end{align*}
  \vspace*{-5mm}\begin{align*}
    \smash{\prod_{p\equiv7[8]}(1-p^{-2})^{-1}}
    =1.&02515\,03739\,25759\,17991\,61954\,35560\,94158\,79433\,11002\,76024
    \\&41530\,69566\,94982\,17644\,97960\,41007\,90076\,26943\,14236\,43529\cdots
  \end{align*}
  }
\subsection*{Modulo 9}
{\footnotesize
  \begin{align*}
    \smash{\prod_{p\equiv1[9]}(1-p^{-2})^{-1}}
    =1.&00403\,38350\,51288\,79798\,24781\,19924\,74748\,94825\,22895\,79877
    \\&28822\,86701\,42359\,63409\,37977\,93839\,33608\,94316\,94860\,37141\cdots
  \end{align*}
  \vspace*{-5mm}\begin{align*}
    \smash{\prod_{p\equiv2,5[9]}(1-p^{-2})^{-1}}
    =1.&40783\,70719\,96538\,05093\,52684\,03433\,79823\,18382\,56159\,80878
    \\&18858\,21039\,93308\,74959\,08486\,21687\,68292\,75777\,90984\,34896\cdots
  \end{align*}
  \vspace*{-5mm}\begin{align*}
    \smash{\prod_{p\equiv4,7[9]}(1-p^{-2})^{-1}}
    =1.&02986\,05876\,77826\,18491\,88642\,35135\,21663\,16312\,01666\,87293
    \\&15881\,63094\,56123\,55333\,65628\,89969\,28513\,96515\,60005\,36245\cdots
  \end{align*}
  \vspace*{-5mm}\begin{align*}
    \smash{\prod_{p\equiv8[9]}(1-p^{-2})^{-1}}
    =1.&00442\,33235\,64550\,15978\,66082\,58390\,58205\,39661\,19672\,30788
    \\&17744\,79626\,23017\,18753\,96410\,76663\,34579\,95134\,16501\,66760\cdots
  \end{align*}
  }
\subsection*{Modulo 11}
{\footnotesize
  \begin{align*}
    \smash{\prod_{p\equiv1[11]}(1-p^{-2})^{-1}}
    =1.&00232\,82408\,97736\,52733\,78057\,92469\,42582\,04345\,78064\,14879
    \\&23124\,99895\,44150\,38255\,72926\,07516\,98484\,87460\,03110\,08712\cdots
  \end{align*}
  \vspace*{-5mm}\begin{align*}
    \smash{\prod_{p\equiv2,6,7,8[11]}(1-p^{-2})^{-1}}
    =1.&38240\,11448\,05788\,71773\,39824\,35954\,70441\,91351\,16435\,84157
    \\&13863\,06101\,70250\,01900\,59181\,34321\,25138\,72741\,06748\,64687\cdots
  \end{align*}
  \vspace*{-5mm}\begin{align*}
    \smash{\prod_{p\equiv3,4,5,9[11]}(1-p^{-2})^{-1}}
    =1.&17640\,19224\,41514\,71776\,56838\,81699\,54785\,03151\,42210\,45715
    \\&72819\,38133\,44304\,81040\,93008\,74341\,67383\,61950\,21979\,26318\cdots
  \end{align*}
  \vspace*{-5mm}\begin{align*}
    \smash{\prod_{p\equiv10[11]}(1-p^{-2})^{-1}}
    =1.&00079\,37707\,14740\,00680\,22327\,79981\,38075\,30993\,79972\,81556
    \\&86828\,01966\,59824\,89326\,65924\,56171\,20791\,11742\,28212\,98769\cdots
  \end{align*}
  }

\subsection*{Modulo 12}
{\footnotesize
  \begin{align*}
    \smash{\prod_{p\equiv1[12]}(1-p^{-2})^{-1}}
    =1.&00761\,32452\,14144\,96616\,93493\,12247\,73229\,37895\,47142\,90433
    \\&17666\,43368\,44819\,49208\,97861\,01855\,78530\,60579\,11129\,80649\cdots
  \end{align*}
  \vspace*{-5mm}\begin{align*}
    \smash{\prod_{p\equiv5[12]}(1-p^{-2})^{-1}}
    =1.&04820\,19036\,00769\,93683\,49374\,34895\,79267\,34804\,13674\,49481
    \\&52581\,07376\,14495\,24161\,71571\,43788\,23594\,04990\,88566\,94968\cdots
  \end{align*}
  \vspace*{-5mm}\begin{align*}
    \smash{\prod_{p\equiv7[12]}(1-p^{-2})^{-1}}
    =1.&02620\,21468\,31233\,70070\,72018\,66966\,36157\,23611\,09321\,31334
    \\&95148\,10400\,66496\,54603\,29393\,86454\,19299\,91782\,63867\,91609\cdots
  \end{align*}
  \vspace*{-5mm}\begin{align*}
    \smash{\prod_{p\equiv11[12]}(1-p^{-2})^{-1}}
    =1.&01177\,86368\,50332\,58370\,51194\,10267\,33127\,80584\,01230\,89520
    \\&87028\,35959\,40756\,15016\,41704\,56300\,54442\,19591\,32980\,62727\cdots
  \end{align*}
  }
\subsection*{Modulo 13}
{\footnotesize
  \begin{align*}
    \smash{\prod_{p\equiv1[13]}(1-p^{-2})^{-1}}
    =
    1.&00065\,68661\,98289\,66605\,74722\,84730\,77197\,91777\,00717\,07399
    \\&33554\,44837\,12988\,36602\,52536\,84343\,79642\,73590\,88077\,31673\cdots
  \end{align*}
   \begin{align*}
    \smash{\prod_{p\equiv2,6,7,11[5]}(1-p^{-2})^{-1}}
    =1.&38005\,21671\,19142\,93623\,73358\,95833\,59312\,88490\,63922\,76216
     \\&00813\,27801\,96170\,83570\,07037\,00666\,02382\,19997\,07055\,85939\cdots
  \end{align*}
  \vspace*{-5mm}\begin{align*}
    \smash{\prod_{p\equiv3,9[13]}(1-p^{-2})^{-1}}
    =1.&12706\,12738\,77030\,37596\,05291\,90459\,70008\,03562\,53668\,12081
    \\&48604\,51380\,13290\,89754\,69987\,12664\,24897\,64722\,52303\,29593\cdots
  \end{align*}
   \begin{align*}
    \smash{\prod_{p\equiv4,10[13]}(1-p^{-2})^{-1}}
    =1.&00628\,51383\,85264\,35654\,79220\,78630\,88874\,03212\,24553\,50607
     \\&59162\,40959\,77321\,01204\,89381\,53735\,74182\,12805\,59112\,51752\cdots
  \end{align*}
   \begin{align*}
    \smash{\prod_{p\equiv5,8[13]}(1-p^{-2})^{-1}}
    =1.&04384\,79529\,58163\,48325\,64453\,12135\,62867\,13038\,05109\,49630
     \\&56435\,71738\,46465\,77456\,29690\,71263\,29350\,03766\,17988\,29979\cdots
  \end{align*}
  \vspace*{-5mm}\begin{align*}
    \smash{\prod_{p\equiv12[13]}(1-p^{-2})^{-1}}
    =1.&00019\,47228\,43353\,09720\,12251\,29852\,70839\,19867\,65951\,93000
    \\&49665\,62593\,02690\,92410\,34974\,82067\,06364\,88262\,34074\,53639\cdots
  \end{align*}
  }
\subsection*{Modulo 15}
{\footnotesize
  \begin{align*}
    \smash{\prod_{p\equiv1[15]}(1-p^{-2})^{-1}}
    =1.&00148\,97422\,73492\,93695\,62022\,82152\,29804\,06202\,71822\,24183
    \\&85046\,92061\,06460\,33370\,47461\,16170\,34094\,66709\,13158\,03303\cdots
  \end{align*}
  \vspace*{-5mm}\begin{align*}
    \smash{\prod_{p\equiv2,8[15]}(1-p^{-2})^{-1}}
    =1.&34246\,04551\,54995\,30799\,30100\,63345\,72665\,24298\,78723\,72380
    \\&96524\,03928\,73058\,62457\,83670\,07480\,09151\,10334\,06933\,31380\cdots
  \end{align*}
  \vspace*{-5mm}\begin{align*}
    \smash{\prod_{p\equiv4[15]}(1-p^{-2})^{-1}}
    =1.&00317\,84700\,07976\,58539\,76886\,54009\,35749\,55893\,69169\,67588
    \\&37351\,26980\,45622\,46578\,84368\,96080\,28447\,94669\,19055\,69351\cdots
  \end{align*}
  \vspace*{-5mm}\begin{align*}
    \smash{\prod_{p\equiv7,13[15]}(1-p^{-2})^{-1}}
    =1.&02920\,54524\,88970\,30487\,46169\,68199\,34620\,53972\,85734\,20801
    \\&87576\,81344\,73863\,39397\,51683\,30560\,76995\,20714\,09590\,99521\cdots
  \end{align*}
  \vspace*{-5mm}\begin{align*}
    \smash{\prod_{p\equiv11[15]}(1-p^{-2})^{-1}}
    =1.&00941\,13977\,70415\,34074\,11140\,07967\,71715\,31828\,38502\,83487
    \\&41065\,68439\,10926\,98429\,51008\,47969\,06005\,15885\,02338\,55701\cdots
  \end{align*}
  \vspace*{-5mm}\begin{align*}
    \smash{\prod_{p\equiv14[15]}(1-p^{-2})^{-1}}
    =1.&00177\,62082\,89544\,73626\,10915\,43079\,96283\,15610\,57061\,98467
    \\&19519\,14691\,39870\,02036\,75682\,26376\,90944\,75824\,69831\,96091\cdots
  \end{align*}
  }
\subsection*{Modulo 16}
{\footnotesize
  \begin{align*}
    \smash{\prod_{p\equiv1[16]}(1-p^{-2})^{-1}}
    =1.&00378\,12963\,11174\,37714\,94711\,72280\,61816\,45658\,26785\,28441
    \\&57268\,63521\,48911\,54134\,99502\,87194\,19254\,71100\,10645\,46873\cdots
  \end{align*}
  \vspace*{-5mm}\begin{align*}
    \smash{\prod_{p\equiv3,11[16]}(1-p^{-2})^{-1}}
    =1.&13941\,87771\,08211\,51502\,70589\,30773\,34020\,88725\,59961\,09629
    \\&48302\,25821\,27411\,02101\,65577\,60742\,91446\,59374\,91512\,33349\cdots
  \end{align*}
  \vspace*{-5mm}\begin{align*}
    \smash{\prod_{p\equiv5,13[16]}(1-p^{-2})^{-1}}
    =1.&05109\,99849\,42183\,30793\,68775\,56006\,33505\,68012\,01018\,45817
    \\&85080\,59912\,94207\,39729\,30485\,58783\,38889\,50479\,59255\,34495\cdots
  \end{align*}
  \vspace*{-5mm}\begin{align*}
    \smash{\prod_{p\equiv7[16]}(1-p^{-2})^{-1}}
    =1.&02325\,48781\,97407\,08067\,95776\,68614\,06977\,00372\,89157\,54600
    \\&19844\,97929\,83355\,91253\,99909\,55714\,70317\,40567\,85934\,05044\cdots
  \end{align*}
  \vspace*{-5mm}\begin{align*}
    \smash{\prod_{p\equiv9[16]}(1-p^{-2})^{-1}}
    =1.&00104\,97991\,21471\,31637\,83963\,95210\,10070\,68052\,00181\,57035
    \\&98663\,81304\,47589\,89310\,55217\,86340\,51978\,44383\,63621\,58656\cdots
  \end{align*}
  \vspace*{-5mm}\begin{align*}
    \smash{\prod_{p\equiv15[16]}(1-p^{-2})^{-1}}
    =1.&00185\,24179\,73996\,13159\,93578\,02219\,51678\,26622\,68517\,41444
    \\&99996\,30754\,09303\,19958\,16127\,21985\,97936\,04820\,77136\,34947\cdots
  \end{align*}
  }

\subsection*{Some notes on timing}

We tried several large computations to get an idea of the limitations
of our script, with the uniform choice $P=300$ and asking for
100~decimal digits. Since we did not run each computations hundred
times to get an average timing, this table has to be taken with a pinch
of salt. We present relative timing, knowing that the computation with
$q=3$, $q=4$ or $q=4$ took about a tenth of a second.
\vspace*{3mm}

\noindent{\footnotesize
\begin{tabular}{|r||r|r|r|r||r|}
  \hline
  $q$&$\varphi(q)$&$\#d_i's$&$|G^\sharp|$&$r$&\parbox{30pt}{relative\\time(ms)}\\
  \hline
  3& 2& 5& 2& 5& 1 \\
4& 2& 5& 2& 5& 1 \\
5& 4& 19& 3& 5& 1 \\
7& 6& 28& 4& 5& 3.2 \\
8& 4& 5& 4& 5& 2.2 \\
9& 6& 28& 4& 5& 3.2 \\
11& 10& 15& 4& 5& 30 \\
12& 4& 5& 4& 5& 2.5 \\
13& 12& 55& 6& 5& 4.4 \\
15& 8& 19& 6& 5& 2 \\
16& 8& 19& 6& 5& 1.6 \\
17& 16& 30& 5& 5& 42 \\
19& 18& 34& 6& 5& 93 \\
20& 8& 19& 6& 5& 2 \\
21& 12& 28& 8& 5& 7 \\
23& 22& 9& 4& 5& 100 \\
24& 8& 5& 8& 5& 5 \\
25& 20& 32& 6& 5& 98 \\
27& 18& 34& 6& 5& 92 \\
28& 12& 28& 8& 5& 6.5 \\
29& 28& 31& 6& 5& 175 \\
31& 30& 54& 8& 5& 343 \\
32& 16& 27& 8& 5& 25 \\
33& 20& 15& 8& 5& 65 \\
35& 24& 55& 12& 5& 96 \\
36& 12& 28& 8& 5& 6.5 \\
37& 36& 61& 9& 5& 350 \\
39& 24& 55& 12& 5& 99 \\
40& 16& 19& 12& 5& 4.6 \\
41& 40& 40& 8& 5& 424 \\
43& 42& 40& 8& 5& 654 \\
44& 20& 15& 8& 5& 652 \\
45& 24& 55& 12& 5& 95 \\
47& 46& 6& 4& 5& 394 \\
48& 16& 19& 12& 5& 4.8 \\
49& 42& 40& 8& 5& 665 \\
  51& 32& 30& 10& 5& 101 \\
  \hline
\end{tabular}\quad
\begin{tabular}{|r||r|r|r|r||r|}
  \hline
  $q$&$\varphi(q)$&$\#d_i's$&$|G^\sharp|$&$r$&\parbox{30pt}{relative\\time(ms)}\\
  \hline
52& 24& 55& 12& 5& 102 \\
53& 52& 23& 6& 5& 675 \\
55& 40& 32& 12& 5& 250 \\
56& 24& 28& 16& 5& 15 \\
57& 36& 34& 12& 5& 222 \\
59& 58& 6& 4& 5& 675 \\
60& 16& 19& 12& 5& 5 \\
61& 60& 84& 12& 5& 1468 \\
63& 36& 28& 20& 5& 23 \\
64& 32& 30& 10& 5& 96 \\
65& 48& 55& 20& 5& 260 \\
67& 66& 32& 8& 5& 155 \\
68& 32& 30& 10& 5& 97 \\
69& 44& 9& 8& 5& 240 \\
71& 70& 24& 8& 5& 1850 \\
72& 24& 28& 16& 5& 15 \\
73& 72& 72& 12& 5& 1643 \\
75& 40& 32& 12& 5& 237 \\
76& 36& 34& 12& 5& 219 \\
77& 60& 54& 16& 5& 855 \\
79& 78& 32& 8& 5& 2312 \\
80& 32& 19& 20& 5& 12 \\
81& 54& 35& 8& 5& 871 \\
83& 82& 5& 4& 5& 1527 \\
84& 24& 28& 16& 5& 15 \\
85& 64& 30& 18& 5& 257 \\
87& 56& 31& 12& 5& 441 \\
88& 40& 15& 16& 5& 157 \\
89& 88& 31& 8& 5& 2058 \\
91& 72& 55& 30& 5& 464 \\
92& 44& 9& 8& 5& 241 \\
93& 60& 54& 16& 5& 866 \\
95& 72& 61& 18& 5& 915 \\
96& 32& 27& 16& 5& 61 \\
97& 96& 70& 12& 5& 3371 \\
99& 60& 54& 16& 5& 855 \\
100& 40& 32& 12& 5& 236 \\
  \hline
\end{tabular}
}
\vspace*{3mm}

\noindent
This table shows that the value of $\varphi(q)$ is the main
determinant of the time needed. The column with the tag ``$\#d_i's$"
contains the number of tuples $(d_1,\cdots, d_v)$ in the main formula.

\noindent\parbox[t]{0.4\linewidth}{~\par\indent Here is now a shorter table when asking 1000 decimal digits still with
$P=300$. The time needed is still very decent.}
\qquad{\footnotesize
\begin{tabular}[t]{|r||r|r|r|r||r|}
  \hline
  $q$&$\varphi(q)$&$\#d_i's$&$|G^\sharp|$&$r$&time(ms)\\
  \hline
  3& 2& 8& 2& 8& 3708  \\
4& 2& 8& 2& 8& 3226  \\
5& 4& 87& 3& 8& 7067  \\
7& 6& 249& 4& 8& 29421  \\
8& 4& 8& 4& 8& 6423  \\
9& 6& 249& 4& 8& 29267  \\
11& 10& 96& 4& 8& 56001  \\
12& 4& 8& 4& 8& 7264  \\
13& 12& 716& 6& 8& 87480  \\
15& 8& 87& 6& 8& 14021  \\
  \hline
\end{tabular}
}

\vspace*{3mm}
When asking for 5000 decimal digits and only $q=3$, it took about 16
minutes (with $P=500$) to get an answer,
which essentially sets the horizon of the present method.

\bibliographystyle{plain}

\begin{thebibliography}{10}

\bibitem{Finch*03}
Steven~R. Finch.
\newblock {\em Mathematical constants}, volume~94 of {\em Encyclopedia of
  Mathematics and its Applications}.
\newblock Cambridge University Press, Cambridge, 2003.

\bibitem{Fouvry-Levesque-Waldschmidt*18}
\'{E}tienne Fouvry, Claude Levesque, and Michel Waldschmidt.
\newblock Representation of integers by cyclotomic binary forms.
\newblock {\em Acta Arith.}, 184(1):67--86, 2018.

\bibitem{Gourdon-Sebah*10}
X.~Gourdon and P.~Sebah.
\newblock Constants from number theory.
\newblock {\em
  \url{http://numbers.computation.free.fr/Constants/constants.html}}, 2010.
\newblock
  \url{http://numbers.computation.free.fr/Constants/Miscellaneous/constantsNumTheory.ps}.

\bibitem{Ireland-Rosen*90}
Kenneth Ireland and Michael Rosen.
\newblock {\em A classical introduction to modern number theory}, volume~84 of
  {\em Graduate Texts in Mathematics}.
\newblock Springer-Verlag, New York, second edition, 1990.

\bibitem{Lal*67}
M.~Lal.
\newblock Primes of the form {$n^{4}+1$}.
\newblock {\em Math. Comp.}, 21:245--247, 1967.

\bibitem{Moree*00}
P.~Moree.
\newblock Approximation of singular series constant and automata. with an
  appendix by gerhard niklasch.
\newblock {\em Manuscripta Matematica}, 101(3):385--399, 2000.

\bibitem{Moree-Osburn*06}
P.~Moree and R.~Osburn.
\newblock Two-dimensional lattices with few distances.
\newblock {\em Enseign. Math. (2)}, 52(3-4):361--380.

\bibitem{Moree*04b}
Pieter Moree.
\newblock On the average number of elements in a finite field with order or
  index in a prescribed residue class.
\newblock {\em Finite Fields Appl.}, 10(3):438--463, 2004.

\bibitem{Niven-Zuckerman-Montgonemry*91}
Ivan Niven, Herbert~S. Zuckerman, and Hugh~L. Montgomery.
\newblock {\em An introduction to the theory of numbers}.
\newblock John Wiley \& Sons, Inc., New York, fifth edition, 1991.

\bibitem{OEIS}
{OEIS Foundation Inc.}
\newblock {\em The On-Line Encyclopedia of Integer Sequence}, 2019.
\newblock \url{http://oeis.org/}.

\bibitem{Rota*64a}
Gian-Carlo Rota.
\newblock On the foundations of combinatorial theory. {I}. {T}heory of
  {M}\"{o}bius functions.
\newblock {\em Z. Wahrscheinlichkeitstheorie und Verw. Gebiete}, 2:340--368
  (1964), 1964.

\bibitem{Serre*70}
Jean-Pierre Serre.
\newblock {\em Cours d'arithm\'{e}tique}, volume~2 of {\em Collection SUP: ``Le
  Math\'{e}maticien''}.
\newblock Presses Universitaires de France, Paris, 1970.

\bibitem{Shanks*64}
D.~Shanks.
\newblock On maximal gaps between successive primes.
\newblock {\em Math. Comp.}, 18:646--651, 1964.

\bibitem{Shanks*64b}
D.~Shanks.
\newblock The second-order term in the asymptotic expansion of {$B(x)$}.
\newblock {\em Math. Comp.}, 18:75--86, 1964.

\bibitem{Shanks*60}
Daniel Shanks.
\newblock On the conjecture of {H}ardy \& {L}ittlewood concerning the number of
  primes of the form {$n^{2}+a$}.
\newblock {\em Math. Comp.}, 14:320--332, 1960.

\bibitem{Shanks*61}
Daniel Shanks.
\newblock On numbers of the form {$n^{4}+1$}.
\newblock {\em Math. Comput.}, 15:186--189, 1961.

\bibitem{Shanks*67}
Daniel Shanks.
\newblock Lal's constant and generalizations.
\newblock {\em Math. Comp.}, 21:705--707, 1967.

\bibitem{Witt*37}
E.~Witt.
\newblock Treue {D}arstellung {L}iescher {R}inge.
\newblock {\em J. Reine Angew. Math.}, 177:152--160, 1937.

\end{thebibliography}

\end{document}